\documentclass[12pt]{amsart}
\usepackage{calc,pifont}
\usepackage{graphicx,adjustbox}
\usepackage{mathrsfs}
\usepackage{amssymb,amsbsy}
\usepackage{amscd}
\usepackage{enumerate}
\usepackage{setspace}
\usepackage{tikz}
\usepackage{bm}
\usepackage{cite}
\usepackage{graphicx}
\usepackage{hyperref}
\allowdisplaybreaks
\usepackage{array}

\usepackage{kpfonts}
\usepackage{stackengine}
\usepackage{calc}
\newlength\shlength
\newcommand\xshlongvec[2][0]{\setlength\shlength{#1pt}%
  \stackengine{-5.6pt}{$#2$}{\smash{$\kern\shlength%
    \stackengine{7.55pt}{$\mathchar"017E$}%
      {\rule{\widthof{$#2$}}{.57pt}\kern.4pt}{O}{r}{F}{F}{L}\kern-\shlength$}}%
      {O}{c}{F}{T}{S}}
  
\usepackage[letterpaper]{geometry}
\usepackage[numbers]{natbib}

\usepackage{amssymb}
\usepackage[utf8]{inputenc}
\usepackage{amsmath}
\usepackage{amsthm}
\usepackage{tikz}
\usepackage{epstopdf}
\usepackage{epsfig}

\usepackage{qtree}
\usepackage{url}

\newcommand{\ben}[1]{{\leavevmode\color{green}{#1}}}
\newcommand{\rv}[1]{{\leavevmode\color{black}{#1}} }

\usepackage[rflt]{floatflt} 
\usepackage[section]{placeins} 

\newtheorem{theorem}{Theorem}[section]
\newtheorem{lemma}[theorem]{Lemma}
\newtheorem{proposition}[theorem]{Proposition}
\newtheorem{definition}[theorem]{Definition}
\newtheorem{remark}[theorem]{Remark}
\newtheorem{corollary}[theorem]{Corollary}

\newcommand{\erre}{\mathbb{R}}

\newcommand{\com}{\mathbb{C}}

\renewcommand{\Re}{\operatorname{Re}}

\begin{document}

\title[Delay-dependent and delay-independet Counot duopoly model]{Delay-dependent and delay-independent stability  of Cournot duopoly model with tax evasion and time-delay}                      
%
%



\author[B.A. Itz\'a-Ortiz]{Benjam\'in A.~Itz\'a-Ortiz}

\author[R. Villafuerte-Segura]{Ra\'ul Villafuerte-Segura}

\author[E. Alvarado-Santos]{Eduardo Alvarado-Santos}

\address{Universidad Aut\'onoma del Estado de Hidalgo, Mexico}

%

\begin{abstract}
In this paper a stability analysis for a Cournot duopoly model with tax evasion and time-delay in a continuous-time framework is presented. The mathematical model under consideration follows a gradient dynamics approach, is nonlinear and four-dimensional with  state variables given by the production and declared revenue of each competitor. We prove that both the marginal cost rate and time delay play roles as bifurcation parameters. More precisely, if the marginal cost rate lies in certain closed interval then the equilibrium point is delay-independent stable, otherwise it is  delay-dependent stable and a Hopf bifurcation necessarily occurs.  Some numerical simulations are presented in order to confirm the proposed theoretical results and illustrate the effect of the bifurcation parameters on model stability.
\end{abstract}

%


\maketitle

\section{Introduction}\label{Intro}
Bifurcation theory is a mathematical field focused on studying the qualitative variations of the behavior occurring in a family of solutions of a given system of differential equations. Recently, it has become a field  of major involvement in  areas such as engineering, physics, chemistry, economics and biology, among others \cite{ fussmann2000crossing, gori2015continuous, hirsch2012differential}. A bifurcation is said to  occur when an infinitesimal variation in the value of a parameter of a nonlinear system causes a qualitative or topological change on the corresponding solutions of the system. This qualitative change, in many cases, refers to a change on the stability of the fixed point, the appearance or disappearance of a fixed point or the creation or annihilation of a periodic orbit. The parameters that cause these changes are known as bifurcation parameters and the values where the changes occur are known as the bifurcation points. In some dynamical systems, the presence of bifurcations often preludes the unveiling  of chaos, or vice versa \cite{sanjuandinamica}. Chaos is the  denomination of the branch of mathematics that pursues the unravelment of certain types  of dynamical systems exhibiting unpredictable behavior \cite{strogatz:2000}. There are different types of bifurcations that can be present in  a dynamical system. In this paper, the analysis is focused on Hopf bifurcations \cite{poincare1885equilibre}. Although there is a large amount of literature that addresses this topic,  little research has been done on  Hopf bifurcations for time-delay systems.
Bifurcation analysis for time-delay systems is found in the literature on problems from diverse areas such as  economics \cite{neamctu2010deterministic}, finance \cite{jun2001study} and biology \cite{fussmann2000crossing}, among others. 

A large number of topics in theoretical economics can be endorsed with a rigorous mathematical framework by analyzing their corresponding mathematical models. Notably, examples of these analyses are the formalization of imbalance models of economic cycles, the evolutionary models of financial markets, models that focus on the study of the dynamic behavior of firms in world markets, among others, \cite{gori2015continuous}. In recent years, oligopoly models have been receiving increasing attention, both by economists and mathematicians. By economists, since the behavior of these models plays a relevant role in theoretical economics \cite{GR}, and by  mathematicians, as the mathematical models ensuing, despite their complexity,  yield interesting  examples of chaos \cite{A1} and Hopf bifurcations \cite{neamctu2010deterministic}. In this study, a Cournot duopoly model with tax evasion and time-delay  is addressed.

The Cournot duopoly model is a classic example in game theory \cite{4, 8}. A duopoly is a market where two firms sell the same product to a large number of consumers. The first study of a duopoly is due to Antoine Augustin Cournot \cite{cournot1897}, who in 1838 proposed that firms should adjust production levels in such a way that each of them maximizes its profits taking into account the production of the rival firm.
Some studies that address this topic are given in \cite{A3,A2,A1}.  
While duopoly models may be regarded as dynamical systems on two variable, in \cite{GR} a Cournot duopoly model with tax evasion was introduced, rendering the corresponding study of the dynamical system to one with four variables, thus increasing its complexity. In \cite{neamctu2010deterministic}, the author presented 
a Cournot duopoly model with tax evasion  where a Hopf bifurcation occurred with variations of time-delay; however, no condition for the existence of such bifurcation was given. In \cite{pecora2018heterogenous}, an analysis of a heterogenous Cournot duopoly with delay dynamics is given. Here, the mathematical model is two-dimensional with state variables being the  quantities which enter the market from the two firms. Also, stability switching curves and  numerical simulations are provided to illustrate the theoretical findings and to show how the delays affect the dynamic behavior.

In this paper a stability analysis of a Cournot duopoly model with tax evasion an time-delay continuous-time framework is presented. Here, the mathematical model is of  four-dimensional and considers the  quantities which enter the market as well as the declared revenues from each competitor. As a consequence of \rv{our analyisis}, conditions  to determine delay-dependent and delay-independent stability of \rv{the Cournot} model are given. This allows determining restrictions for the existence of limit cycles and Hopf bifurcations. Also, it is shown that\rv{,  notwithstanding the values of many other parameter such as the tax rate and the probabilities of being caught evading taxes, it is} the marginal cost rate of firms \rm{which turns out to be the} decisive factor in determining stability switching in the Cournot duopoly model under delay variations, as well as stability no-switching under any delay variation.

The paper is organized as follows. Some preliminary results concerning to stability of time-delay system and Cournot duopoly nonlinear model with time-delay are presented in Section 2. The stability analysis of Cournot duopoly nonlinear model to obtain the bifurcation parameters, limit cycles and  Hopf bifurcations is proposed in Section 3. In Section 4, the implementation and validation of the previous theoretical results obtained are given. Finally, some concluding remarks are made in Section 5.

\section{Preliminary results }

\subsection{Stability of time delay systems }
In this section, some results concerning the stability of time\ben{-}delay systems are given.

\noindent
Some necessary notation is given first. Consider a time delay nonlinear system of the form
\begin{align}\label{NonlinearSystem}
   \dfrac{d\vec{x}}{dt}= G(\vec{x},\vec{x}\sb\tau),
\end{align}
where $G(\vec{x},\vec{x}\sb\tau)=(g_{1}(\vec { x }(t), \vec { x }_{\tau}(t))\ g_{2}(\vec { x }(t), \vec { x }_{\tau}(t)),\dots,g_{n}(\vec { x }(t),\vec { x }_{\tau}(t)))^{\intercal}$, $\vec{x}=(x_1(t)\ x_2(t)\ z_1(t)\ z_2(t))^{\intercal}$,
$\vec{x}\sb\tau=\vec{x}(t-\tau)=({x_1}\sb\tau(t)\ {x_2}\sb\tau(t)\ {z_1}\sb\tau(t)\ {z_2}\sb\tau(t))^{\intercal}=\left( x_1(t-\tau)\ x_2(t-\tau)\ z_1(t-\tau)\ z_2(t-\tau)\right)^{\intercal}$.

\medskip
\noindent
Now, a equilibrium ${ \vec { x }  }^{ * }={ (x }_{ 1 }^{ * },{ x }_{ 2 }^{ * },\dots ,{ x }_{ n }^{ * })$ is the one that satisfies
$G({ \vec { x }  }^{ * },{ \vec { x }  }_{ \tau  }^{ * })=G({ \vec { x }  }^{ * },{ \vec { x }  }^{ * })=0.$ Thus, the linearization of (\ref{NonlinearSystem}) at the equilibrium point is 
\begin{equation}\label{GralLineal}
\dfrac{d \vec {x}(t)}{dt}=A\vec { x } +B\vec { x }_{\tau},
\end{equation}
where
\begin{equation*}
\begin{array}{ll}
A=\left( \begin{matrix} { D }_{ 1 }{ g }_{ 1 } & { D }_{ 2 }{ g }_{ 1 } & \ldots & { D }_{ n }{ g }_{ 1 } \\ { D }_{ 1 }{ g }_{ 2 } & { D }_{ 2 }{ g }_{ 2 } & \ldots & { D }_{ n }{ g }_{ 2 } \\ \vdots  &  & \ddots  & \vdots  \\ { D }_{ 1 }{ g }_{ n } & { D }_{ 2 }{ g }_{ n } & \dots  & { D }_{ n }{ g }_{ n } \end{matrix} \right),
& \rv{B=\left( \begin{matrix} { D }_{ 1_{\tau} }{ g }_{ 1_{\tau} } & { D }_{ 2_{\tau} }{ g }_{ 1_{\tau} } & \ldots & { D }_{ n_{\tau} }{ g }_{ 1_{\tau} } \\ { D }_{1_{\tau} }{ g }_{ 2_{\tau} } & { D }_{ 2_{\tau} }{ g }_{ 2_{\tau} } & \ldots & { D }_{ n_{\tau} }{ g }_{ 2_{\tau} } \\ \vdots  &  & \ddots  & \vdots  \\ { D }_{ 1_{\tau} }{ g }_{ n_{\tau} } & { D }_{ 2_{\tau} }{ g }_{ n_{\tau} } & \dots  & { D }_{n_{\tau} }{ g }_{ n_{\tau} } \end{matrix} \right) } 
\end{array}
\end{equation*}
are constant systems matrices in $\erre^{n\times n}$, $\tau\in\erre^{+}$ is a delay, $\psi:[-\tau,0]\rightarrow\mathfrak{C}$ is the initial condition, 
$C([-\tau,0],\mathbb{R}^{n})$ is Banach space of continuous vector functions mapping the interval $[-\tau,0]$ into $\erre^{n}$ with the standard uniform norm
$\|\psi\|_{\tau}:=\max_{\theta\in[-\tau,0]}\|\psi(\theta)\|$. For any initial condition the system (\ref{GralLineal}) admits the unique solution $x(t,\psi)$ defined on $[-\tau,\infty]$ and $x_{t}(\psi):=\{x(t+\theta,\psi )\subset x(t,\psi): \theta \in \lbrack -\tau,0] \}$ is the state vector. Here, ${ D }_{ l }{ g }_{ j }=k_l\frac{\partial }{\partial x_l}  \left( \frac{\partial P_j}{\partial x_j}\right)$, and \rv{${ D }_{ l_{\tau} }{ g }_{ j_{\tau} }=k_l\frac{\partial }{\partial x_{l_{\tau}}} \left( \frac{\partial P_j}{\partial x_{j_\tau}}\right)$}; $l,\ j=1,\ldots,n$.

\medskip
\noindent
The above system is know as linear time invariant systems (LTI) with time-delay or LTI system with time-delay, and its quasi-polynomial characteristic is of the form
\begin{equation}\label{GralQuasi}
q(\lambda,\tau)=det(A+e^{-\lambda\tau}B-\lambda I)=0.
\end{equation}

\begin{definition}\cite{cooke1986zeroes}\label{DefHurwitz}
The LTI systems with time-delay (\ref{GralLineal}) is asymptotically stable if all zeros of quasi-polynomial (\ref{GralQuasi}) lie in $\Re\{\lambda\}<0$, $j=1,2,\ldots$.
\end{definition}

\medskip
\noindent
It should be noted that the stability of a system of the form (\ref{NonlinearSystem}) does not always depend on the variations of the parameter $\tau$, that is, the system (\ref{NonlinearSystem}) is stable for arbitrary delay. This condition is known as delay-independent stability criteria. On the other hand, when the stability of the system (\ref{NonlinearSystem}) depends on the variations of $\tau$, this is known as delay-dependent stability criteria.

\medskip
\noindent
Next, the Cournot duopoly model to be studied is introduce.


\subsection{Cournot duopoly mathematical model}

Consider $P_l:\erre^4\rightarrow\erre^+$, $l=1,2$ the profit functions of the two firms given by
\begin{align}\label{ProfitFuction}
P_{l}&\left( x_{1}(t), x_{2}(t), z_{1}(t), z_{2}(t)\right)\notag\\
&= \left(1-q_{l}\right)\left[ x_{l}(t)p(y(t))-C_{l}(x_{l}(t))-\sigma z_{l}(t)\right]+q_{l}\left[(1-\sigma)x_{l}(t)p(y(t))-C_{l}(x_{l}(t))-F(x_{l}(t)p(y(t))-z_{l})\right]\notag\\
 &=(1-q_{l}\sigma)x_{l}(t)p(y(t)) - q_{l}F\left( x_{l}(t)p(y(t))-z_{l}(t)\right)-(1-q_{l})\sigma z_{l}(t) - C_{l}(x_{l}(t)),
\end{align}
where $x_l(t)\in\erre^{+}$, $l=1,2$ are the quantities which enter the market from the two firms, $z_l(t)\in\erre^{+}$, $l=1,2$ are the declared revenues and $y(t):=x_1(t)+x_2(t)\in\erre^{+}$ is a combination of the above variables. In the subsequent, the following is written to reduce notation, $x_l:=x_l(t)$, $z_l=z_l(t)$ and $y:=y(t)$. Also, $p:\erre^+\rightarrow\erre^+$ is the inverse demand function such that is a derivable function with $p'(y)<0$, $\lim _{ y\rightarrow a }{p(y) = 0 }$
and $\lim _{ y\rightarrow 0 }{p(y) =b }$; $a,\ b\in\bar{\erre}^+$, $F(y):\erre^+\rightarrow\erre^+$ is the penalty function such that $F'(y)>0$, $F''(y)>0$, $F(0)=0$,
$C_l:\erre^+\rightarrow\erre^+$, $l=1,2$, is the cost function such that $C_l$ are derivable functions with $C_l'(y)>0$, $C_l''(y)\geq 0$, $\sigma\in[0,1)$ is the government levies an ad valorem tax on each firm’s sales, $q_l\in(0,1)$, $l=1,2$ is the joint probability of being audited and detected, $\sigma z_l$ is the tax bill of firm $l$.

\medskip
\noindent
In (\ref{ProfitFuction}), the first bracketed term  equals the profit of firm $l$ if evasion activities remain undetected, while the second
term represents the profit of firm $l$ in case tax evasion is detected. The model assumes that each firm tends to maximize profits, based on the expectation that its own production and the declared revenue decision will not have an effect on the decisions of its rivals. Therefore, the purpose of the firm is to maximize (\ref{ProfitFuction}) with respect to the output $x_l$ and the declared income $z_l$. Thus, the mathematical optimization problem is given by
\begin{equation}\label{max}
\max_{x_l,z_l}P_l;\quad l=1,2.
\end{equation}
The following proposition, which is similar to  \cite[Proposition~1]{neamctu2010deterministic} except that the functions $F$ and $C_l$ are not yet specified. Compare also with \cite[Proposition~1.1]{itza2012}.
\begin{proposition}\label{PuntosCriticos}
The values of  $x_1\sp\ast$, $x_2\sp\ast$, $z_1\sp\ast$ y $z_2\sp\ast$ which maximizes the profit functions $P_1$ and $P_2$ satisfy the following four equations
\begin{equation}\label{opti}
\begin{split}
\dfrac{\partial P_l}{\partial x_l} &= \left(1-q_{l}\sigma-q_{l}F^{\prime}(x_{l}p(y)-z_{l})\right)\left(p(y)+x_{l}p'(y)\right)-C_{l}'(x_{l})=0,\\
\dfrac{\partial P_l}{\partial z_l} &= -(1-q_{l})\sigma+q_{l}F^{\prime}(x_{l}p(y)-z_{l})=0;\quad l=1,2.
\end{split}
\end{equation}
\end{proposition}
\medskip
\noindent
For the dynamical model,  the time dependent input variable $x_l(t)$ for each firm  is considered. It will be assumed that the variation of $x_l(t)$ with respect to time is proportional to the marginal profit $\frac{\partial P_l}{\partial x_l}$. Similarly, the declared revenue for each firm is considered to be time dependent, $z_l(t)$ and the adjustment of the amount declared is assumed to be proportional  to the marginal profits $\frac{\partial P_l}{\partial z_l}$. However, the second firm is assumed to be a follower of the first, that is to say, the first firm is assumed to enter the market first followed after a delay $\tau$ by the second firm. 

\medskip
\noindent
Thus, Cournot duopoly nonlinear model with tax evasion and time-delay is  
\begin{align}\label{tau}
G(\vec{x},\vec{x}\sb\tau)&=\left(\frac{dx_{ 1 }}{dt}\quad \frac{dx_{ 2 }}{dt}\quad 
\frac{dz_{ 1 }}{dt}\quad \frac{dz_{ 2 }}{dt}\quad\right)^{\intercal}\\
&=\left({k_1} \dfrac{\partial P_1(\vec{x},\vec{x}\sb\tau)}{ \partial x_1}\quad
k_{2}\dfrac{\partial P_2(\vec{x},\vec{x}\sb\tau)}{ \partial x_2}\quad
k_{3}\dfrac{\partial P_1(\vec{x},\vec{x}\sb\tau)}{ \partial z_1}\quad
k_{4}\dfrac{\partial P_2(\vec{x},\vec{x}\sb\tau)}{ \partial z_2}
 \right)^{\intercal}.\notag
\end{align}
where $k_j$, $j=1,\ldots,4$ are constant. Observe that the above nonlinear model is of the form (\ref{NonlinearSystem}). Moreover, the fixed points of the system (\ref{tau}) is precisely the equilibrium points computed in Proposition~\ref{PuntosCriticos}.

\begin{proposition}
The linearization of systems (\ref{tau}) around the equilibrium (\ref{equilibrium}) is a system of the form (\ref{GralLineal}), 
\begin{equation}\label{GralLineal_duopolio}
\dfrac{d \vec {x}}{dt}=A\vec { x } +B\vec { x }_{\tau},
\end{equation}
where
\begin{equation*}
\begin{array}{ll}
A=\begin{pmatrix} { k }_{ 1 }\frac { { \partial  }^{ 2 }{ P }_{ 1 } }{ \partial { x }_{ 1 }^{ 2 } }  & { k }_{ 1 }\frac { { \partial  }^{ 2 }{ P }_{ 1 } }{ \partial { x }_{ 2 }{ \partial x }_{ 1 } }  & { k }_{ 1 }\frac { { \partial  }^{ 2 }{ P }_{ 1 } }{ \partial { z }_{ 1 }\partial { x }_{ 1 } }  & 0 \\ 0  & { k }_{ 2 }\frac { { \partial  }^{ 2 }{ P }_{ 2 } }{ \partial { x }_{ 2 }^{ 2 } }  & 0 & { k }_{ 2 }\frac { { \partial  }^{ 2 }{ P }_{ 2 } }{ \partial { z }_{ 2 }{ \partial x }_{ 2 } }  \\ { k }_{ 3 }\frac { { \partial  }^{ 2 }{ P }_{ 1 } }{ \partial { x }_{ 1 }{ \partial z }_{ 1 } }  & { k }_{ 3 }\frac { { \partial  }^{ 2 }{ P }_{ 1 } }{ \partial { x }_{ 2 }{ \partial z }_{ 1 } }  & { k }_{ 3 }\frac { { \partial  }^{ 2 }{ P }_{ 1 } }{ \partial { z }_{ 1 }^{ 2 } }  & 0 \\ 0  & { k }_{ 4 }\frac { { \partial  }^{ 2 }{ P }_{ 2 } }{ { \partial x }_{ 2 }\partial { z }_{ 2 } }  & 0 & { k }_{ 4 }\frac { { \partial  }^{ 2 }{ P }_{ 2 } }{ \partial { z }_{ 2 }^{ 2 } }  \end{pmatrix}, &
B=\begin{pmatrix} 0  & 0  & 0  & 0 \\ 
{ k }_{ 2 }\frac { { \partial  }^{ 2 }{ P }_{ 2 } }{ \partial { x }_{ 1\tau  }{ \partial x }_{ 2 } }  & 0  & 0 & 0 \\ 
0  & 0 & 0  & 0 \\ 
{ k }_{ 4 }\frac { { \partial  }^{ 2 }{ P }_{ 2 } }{ \partial { x }_{ 1\tau  }{ \partial z }_{ 2 } }  & 0  & 0 & 0  \end{pmatrix}.
\end{array}
\end{equation*}
and its quasi-polynomial is
\begin{align}\label{quasi}
q(\lambda,\tau)=p_{1}(\lambda)p_{2}(\lambda)-e^{-\lambda \tau}(a\lambda-b)(c\lambda-d),
\end{align}
where $p_{1}(\lambda)=\lambda^2- \alpha_1\lambda + \alpha_0$, $p_{2}(\lambda)=\lambda^2- \beta_1\lambda + \beta_0$, with
\begin{align*}
\alpha_0&=\left( k_1 \dfrac{\partial^2P_1}{\partial x_1^2} k_{3} \dfrac{\partial^2 P_1}{ \partial z_1^2}\right) - \left(k_1 \dfrac{\partial^2P_1}{\partial z_1\partial x_1} k_{3} \dfrac{\partial^2 P_1}{\partial {x_1} \partial z_1}\right),\quad
\alpha_1=\left(k_1 \dfrac{\partial^2P_1}{\partial x_1^2}+ k_{3} \dfrac{\partial^2P_1}{\partial z_1^2}  \right),\\
\beta_0&=\left(k_2 \dfrac{\partial^2P_2}{\partial x_2^2} k_{4} \dfrac{\partial^2 P_2}{ \partial z_2^2}\right)-\left( k_2 \dfrac{\partial^2P_2}{\partial z_2\partial x_2} k_{4} \dfrac{\partial^2 P_2}{\partial {x_2} \partial z_2}\right),\quad 
\beta_1= \left(k_2 \dfrac{\partial^2P_2}{\partial x_2^2}+ k_{4} \dfrac{\partial^2P_2}{\partial z_2^2}\right).\\
\end{align*}
Additionally,  
\begin{align*}
a&={ k }_{ 1 }\frac { \partial^{2} { P }_{ 1 } }{ { \partial x }_{ 2 } { \partial x }_{ 1 }},\quad b=\left( { k }_{ 1 }
\frac { { \partial  }^{ 2 }{ P }_{ 1 } }{ { \partial  }x_{ 2 }{ \partial x }_{ 1 } }  { k }_{ 3 }\frac { { \partial  }^{ 2 }{ P }_{ 1 } }{ { \partial { z }_{ 1 }^{ 2 } } }  \right) -\left( { k }_{ 3 }\frac { { \partial  }^{ 2 }{ P }_{ 1 } }{ { \partial  }x_{ 2 }{ \partial z }_{ 1 } } { k }_{ 1 }\frac { { \partial  }^{ 2 }{ P }_{ 1 } }{ { \partial { z }_{ 1 }{ \partial x }_{ 1 } } }  \right),\\
c&={ k }_{ 2 }\frac { { \partial  }^{ 2 }{ P }_{ 2 } }{ { \partial  }x_{ 1 }{ \partial x }_{ 2} },\quad d=\left( { k }_{ 2 }\frac { { \partial  }^{ 2 }{ P }_{ 2 } }{ { \partial  }x_{ 1 }{ \partial x }_{ 2 } } { k }_{ 4 }\frac { { \partial  }^{ 2 }{ P }_{ 2 } }{ { \partial { z }_{ 2 }^{ 2 } } }  \right) -\left( { k }_{ 4 }\frac { { \partial  }^{ 2 }{ P }_{ 2 } }{ { \partial  }x_{ 1 }{ \partial z }_{ 2 } } { k }_{ 2 }\frac { { \partial  }^{ 2 }{ P }_{ 2 } }{ { \partial { z }_{ 2 }{ \partial x }_{ 2 } } }  \right).
\end{align*}
\end{proposition}

\rv{Next, stability of the system (\ref{GralLineal_duopolio}) is studied.}


\section{Stability analysis of the Cournot duopoly model with time-delay}
In this section, an analysis to determine delay-independent and delay-dependent
stability conditions of the Cournot duopoly model with time-delay is presented.

\rv{The stability of system (\ref{GralLineal_duopolio}) is completely determined by the location of the roots of the corresponding characteristic quasi-polynomial. One method to analyze the stability of a quasi-polynomial is the D-partition method proposed by Neimark in \cite{neimark1973d}. What this method proposes is the  study of the space of crossover frequencies $i\omega$-crossing delays. Below, this method is then applied to  quasi-polynomial (\ref{quasi}).
In addition, we will asume the the system is initially stable, that is, stable when $\tau=0$.
\medskip

\noindent
A stable quasi-polynomial loses  stability if some of its roots cross to the open right-half  of  the complex plane. Clearly, the above occurs when the roots cross the imaginary axis. For this, there are two possible cases,  the first is a crossing window on the imaginary axis, $\lambda=\pm i\omega$, where $0\neq\omega\in\erre^+$, the second is a crossing window on the origin $\lambda=0$. In both cases, $\lambda$ must be solution of quasi-polynomial. In general, the crossing window occur under variations of the parameters of a system or quasi-polynomial. A particular case and of great interest to the scientific community since it is closely related to bifurcation theory, it is to find the crossing windows when delay $\tau$ varies. On the one hand, when the stability of a system depends on the value of $\tau$, then it is said that the system is delay-dependent stable, and there will be ranges of values of $\tau$  for which the system is stable and ranges for which it is unstable. On the other hand, when the system is delay-independent stable, then the system is stable for all non-negative values of $\tau$. Next, an analysis of the quasi-polynomial (9) using the mentioned above is performed.

\medskip
\noindent
Consider the change of variable $\lambda=0$ in the quasi-polynomial (\ref{quasi})
\begin{align*}
q(0,\tau)&=\alpha_0\beta_0-bd.
\end{align*}
Clearly, the previous equation does not contribute much about the stabilized analysis of (\ref{quasi}), whence, efforts will be focused when $\lambda= i\omega_0$, where $0\neq\omega_0\in\erre^+$ is solution of polynomial $P(i\omega_0)=0$ given in (\ref{quasi}).
To obtain stability conditions, it is enough to study only roots with a positive imaginary part, so $\lambda=- i\omega_0$ will not be used.
}

\medskip
\noindent
Now, consider the change of variable $\lambda=i\omega$ in the quasi-polynomial (\ref{quasi})
\begin{align}\label{quasi_iw}
q(i\omega,\tau)&=p_{1}(i\omega)p_{2}(i\omega)-e^{-i\omega \tau}(i\omega a-b)(i\omega c-d)\notag\\ &=p_{1}(i\omega)p_{2}(i\omega)-\left(\cos(\omega\tau)-i\sin(\omega\tau) \right)(a\lambda-b)(c\lambda-d)=0,
\end{align}
Note that \rv{$q(i\omega,\tau)=0$ iff  $\emph{Re}\{q(i\omega,\tau) \}=0$ and $\emph{Im}\{q(i\omega,\tau) \}=0$, where
\begin{align*}
\emph{Re}\{q(i\omega,\tau) \} &= \Phi(\omega)+(ac \omega^{2}-db)\cos(\omega \tau)+\omega(ad+bc)\sin(\omega \tau)=0,\\
\emph{Im}\{q(i\omega,\tau) \} &= \Theta(\omega)+\omega(ad  +bc )\cos(\omega \tau)+(db - ac \omega^{2})\sin(\omega \tau)=0,
\end{align*}
with 
\begin{align*}
\Phi(\omega)=&\emph{Re}\{p_{1}(i\omega)p_{2}(i\omega) \}
={\omega}^{4}- \left( \alpha_{{1}}\beta_{{1}}+\alpha_{{0}}+\beta_{{0}} \right) {\omega}^{2}+\alpha_{{0}}\beta_{{0}},\\
\Theta(\omega)=&\emph{Im}\{p_{1}(i\omega)p_{2}(i\omega) \}
= \omega^{3} \left( \alpha_{{1}}+\beta_{{1}} \right) - \left( \alpha_{{0}}\beta_{{1}}+\alpha_{{1}}\beta_{{0}} \right) \omega.
\end{align*}
}
\medskip
\noindent
In other words, $q(i\omega,\tau)=0$ if
\begin{equation*}
\left[ \begin{matrix} \cos(\omega \tau ) \\ \sin(\omega \tau ) \end{matrix} \right] ={ \begin{bmatrix} (ac{ \omega  }^{ 2 }-db) & \omega(ad +bc ) \\ \omega(ad +bc ) & (db - ac \omega^{2}) \end{bmatrix} }^{ -1 }\left[ \begin{matrix} -\Phi(\omega) \\ -\Theta(\omega) \end{matrix} \right] 
\end{equation*}
or
\begin{equation*}
\cos(\omega \tau) = \frac { -\Phi(\omega)\left( ac{ \omega  }^{ 2 }-bd \right) -\Theta(\omega) \left( ad+bc \right)\omega  }{ (c^2\omega^2+d^2)(a^2\omega^2+b^2) },\quad  
\sin(\omega \tau) = { \frac { -\Phi(\omega)\,\left( ad+bc \right) \omega\, +\Theta(\omega)\left( ac{ \omega  }^{ 2 }-bd \right)  }{(c^2\omega^2+d^2)(a^2\omega^2+b^2) }  }.
\end{equation*}
Now, using $\sin^2(\omega \tau)+\cos^2(\omega \tau)=1$, the above is true if
\begin{align}\label{Polynomial_omega}
P(\omega)&=N_{1}^{2}(\omega) +N_{2}^{2}(\omega)-Q^{2}(\omega)\notag\\
&=a_{12}\omega^{12} + a_{10}\omega^{10} + a_{8}\omega^{8}+ a_{6}\omega^{6}+a_{4}\omega^{4}+a_{2}\omega^{2} +a_{0}=0
\end{align}
with $N_1(\omega)= -\Phi(\omega)\left( ac{ \omega  }^{ 2 }-bd \right) -\Theta(\omega)\omega \, \left( ad+bc \right)$,
$N_2(\omega)=-\Phi(\omega)\,\left( ad+bc \right)\omega +\Theta(\omega) \left( ac{ \omega  }^{ 2 }-bd \right)$ and $Q(\omega)=\left( {c}^{2}{\omega}^{2}+{d}^{2} \right)  \left( {a}^{2}{\omega}^{2}+{b}^{2} \right)$. Also,
\begin{align*}
 a_{12}=&\, a^{2}\, c^{2},\\
 a_{10}=&\left(  \left( {\alpha_{{1}}}^{2}+{\beta_{{1}}}^{2}-2 \left(\,\alpha_{{0}}+\,\beta_{{0}} \right) \right) {c}^{2}+{d}^{2} \right) {a}^{2}+{b}^{2}{c}^{2},\\
 a_{8} =& \left( {\alpha_{{1}}}^{2}+{\beta_{{1}}}^{2}  -    2\, \left( \alpha_{{0}}+ \beta_{{0}} \right) \right)  \left( {a}^{2}{d}^{2}+{b}^{2}{c}^{2} \right)+ \left( {\alpha_{{0}}}^{2}+{\beta_{{0}}}^{2}+ \left( 2\,
\alpha_{{0}}-{\alpha_{{1}}}^{2} \right)  \left( 2\,\beta_{{0}}-{\beta_{{1}}}^{2}
 \right)  \right) {a}^{2}{c}^{2}+{b}^{2}{d}^{2}-{a}^{4}{c}^{4},\\
a_{6} =& \left( {a}^{2}{d}^{2}+{b}^{2}{c}^{2} \right)  \left( {\alpha_{{0}}}^{
2}+{\beta_{{0}}}^{2}+ \left(2\,\alpha_{{0}} -{\alpha_{{1}}}^{2}
 \right)  \left( 2\,\beta_{{0}}-{\beta_{{1}}}^{2} \right)  \right)+ \left( {\alpha_{{1}}}^{2}+{\beta_{{1}}}^{2}  -2 \left( \alpha_{{0}}+ \,\beta_{{0}} \right)\right) {b}^{2}{d}^{2}-2\,{a}^{4}{c}^{2}{d}^{2}\\
&-2\,{b}^{2}{c}^{4}{a}^{2}+ \left( {\alpha_{{0}}}^{2}{\beta_{{1}}}^{2}+{
\alpha_{{1}}}^{2}{\beta_{{0}}}^{2} -2 \left( {\alpha_{{0}}}^{2}\beta_{{0}}+ \,\alpha_{{0}}{\beta_{{0}}}^{2} \right)  \right)          {c}^{2}{a}^{2}\\
 a_{4} =& \left( {\alpha_{{0}}}^{2}{\beta_{{0}}}^{2}-4\,{b}^{2}{d}^{2}
 \right) {c}^{2}{a}^{2}+ \left( {\alpha_{{0}}}^{2}+{\beta_{{0}}}^{2}+
 \left( {\alpha_{{1}}}^{2}-2\,\alpha_{{0}} \right)  \left( {\beta_{{1}
}}^{2}-2\,\beta_{{0}} \right)  \right) {d}^{2}{b}^{2},\\
&+ \left( {\alpha_
{{0}}}^{2} \left( {\beta_{{1}}}^{2}-2\,\beta_{{0}} \right) + \left( {
\alpha_{{1}}}^{2}-2\,\alpha_{{0}} \right) {\beta_{{0}}}^{2} \right) 
 \left( {a}^{2}{d}^{2}+{b}^{2}{c}^{2} \right) -{a}^{4}{d}^{4}-{b}^{4}{
c}^{4},\\
a_{2} =&  \left( {\alpha_{{0}}}^{2} \left( {\beta_{{1}}}^{2}-2\,\beta_{{0}}
 \right) + \left( {\alpha_{{1}}}^{2}-2\,\alpha_{{0}} \right) {\beta_{{0
}}}^{2} \right) {b}^{2}{d}^{2}- \left( {a}^{2}{d}^{2}+{b}^{2}{c}^{2}
 \right)  \left( 2\,{b}^{2}{d}^{2}-{\alpha_{{0}}}^{2}{\beta_{{0}}}^{2}
 \right),\\
 a_{0} =&\, {\alpha _{0}}^{2}{\beta _{0}}^{2}{b}^{2}{d}^{2}-{b}^{4}{d}^{4}.
\end{align*}
Thus, the quasi-polynomial (\ref{quasi}) have dominant roots $\lambda_0=\pm i\omega_0$, if there is $\omega_0$ solution of the polynomial $P(\omega_0)$ given in (\ref{Polynomial_omega}). Moreover, the delay for which the above occurs is
\begin{equation}\label{tau0}
{ \tau  }_{ 0 }=\frac { 1 }{ { \omega_0  } } { \tan }^{ -1 }\left(\frac { N_2(\omega_0) }{ N_1(\omega_0) } \right) + \dfrac{n\pi}{\omega_0};\ n=0,\pm 1, \pm2, \ldots
\end{equation}
The following proposition is  a reformulation of \cite[Theorem~6]{neamctu2010deterministic}. It will be useful for the applications later is this paper.

\begin{proposition}\label{theorem_gral}
Consider the Cournot duopoly linear system with time-delay (\ref{GralLineal_duopolio}) stable for $\tau=0$ and its corresponding quasi-polynomial  (\ref{quasi}). Then, the system (\ref{GralLineal_duopolio}) is delay-independent stable if the polynomial $P(\omega)$ given in (\ref{Polynomial_omega}) has no nonzero real roots.

\medskip
\noindent
On the other hand, if there is $0\neq\omega_0\in\erre^+$ such that $P(\omega_0)=0$, then the system (\ref{GralLineal_duopolio}) is delay-dependent stable. Moreover, the Cournot duopoly nonlinear model with time-delay (\ref{tau}) have a Hopf bifurcation occurs at $\tau=\tau_0$ if 
\begin{equation}\label{sign}
\text{sign}\Bigl\{\text{Re}\Bigl\{ \frac{\partial \lambda}{\partial\tau}\Big{|}_{\lambda=i\omega_0}\Bigl\} \Bigl\}=\text{sign}\Bigl\{ \frac{\partial\, \text{Re}\left\{\lambda\right\}}{\partial\tau}  \Big{|}_{\lambda=i\omega_0} \Bigr\}>0.
\end{equation}
Here, $\tau_0$ is given in (\ref{tau0}).
\end{proposition}
\begin{proof}
It is well-know that the stability of Cournot duopoly linear model (\ref{GralLineal_duopolio}) depends on the location of the roots of the polynomial in the complex plane. Suppose that the system (\ref{GralLineal_duopolio}) is stable for $\tau=0$, i.e. all  roots of (\ref{quasi})  lie in the open left-half of complex plane. 

\medskip
\noindent
By taking $\tau$ as a parameter and the continuous movement of the roots under variation of $\tau$. The quasi-polynomial (\ref{quasi}) lose stability if some of its roots cross to open right-half of complex plane. Clearly, the above occurs when the roots cross the imaginary axis, for which, there must first be a crossing window on the imaginary axis, $\lambda_0=\pm i\omega_0$, where $0\neq\omega_0\in\erre^+$ is solution of polynomial $P(\omega_0)=0$ given in (\ref{Polynomial_omega}). Thus,  
the crossing window $\lambda_0$ is guaranteed and is occurs when $\tau=\tau_0$ and the
nonlinear system (\ref{tau}) have a bifurcation occurs at $\tau=\tau_0$.
Second, suppose that suppose that the above is true, i.e. the  $P(\omega_0)=0$ has at least one positive root and this is simple. As $\tau$  increases, stability switches may occur when (\ref{sign}) is met. Therefore, the system (\ref{GralLineal_duopolio})
is delay-dependent stable, see \cite{cooke1986zeroes,michiels2007stability}.

\medskip
\noindent
On the other hand, if there is not a positive root $\omega_0$ such that the polynomial $P(\omega_0)=0$, then there is no crossing window $\lambda_0=\pm i\omega_0$. Therefore, if the system (\ref{GralLineal_duopolio}) is stable at $\tau=0$ it remain stable for all $\tau\geq0$.

\end{proof}


\medskip
\noindent
Note that the above results are for any demand function $p(y)$, penalty function $F(y)$ and cost functions $C_l(y)$. Throughout the rest of the document it is assumed that the previous functions are defined particularly to obtain specific and detailed results.  

\medskip
\noindent
Consider a Cournot duopoly nonlinear model with time-delay of the form (\ref{tau}), the linearization (\ref{GralLineal_duopolio}), the quasi-polynomial (\ref{quasi}) and the polynomial (\ref{Polynomial_omega}). Also, 

\medskip
\noindent
\textbf{Assumption A:} {Let} $p(y)=1/y$, \rv{$F(u)=\frac{1}{2}s\sigma u^2$} and $C_l(x_l)=c_lx_l$,  with $s>0$, $c_l>0$, $l=1,2$, are constants, $y=x_1+x_2$ and \rv{$ u=x_{1}\, p \left(y\right) - z_{1}$}.

\medskip
\noindent
Thus, the profit function of the first firm is 
   \begin{align}\label{P1Retardo} 
P_{1} &( x_{1}, x_{2}, z_{1}, z_{2},{x_{1}}\sb\tau, {x_{2}}\sb\tau, {z_{1}}\sb\tau, {z_{2}}\sb\tau )    = \left(1 - q_1\right) \big(x_{1}\, p (x_1+x_2) - C_{1} (x_{1}) - \sigma z_{1}\bigr)   
 \notag  \\ 
& + \, q_1 \, \biggl(\left(1 - \sigma\right) x_{1}\, p (x_1+x_2) - C_{1} (x_{1})- F \bigl( x_{1}\, p \left(x_1+x_2\right) - z_{1}\bigr)\biggr).  
\end{align}
The revenue of the second firm is then 
  \[
  x_2\,p\left({x_1}\sb\tau+x_2 \right):= x_2(t)\, p\left(x_1(t-\tau)+x_2(t) \right).
  \]
Therefore, the profit function $P_2$ is given by 
\begin{align}\label{P2Retardo}
P_{2} &\left(x_{1}, x_{2}, z_{1}, z_{2},{x_{1}}\sb\tau, {x_{2}}\sb\tau, {z_{1}}\sb\tau, {z_{2}}\sb\tau\right)= \left(1 - q_2\right) \bigl(x_{2}\, p \left({x_1}\sb\tau+x_2\right) - C_{2} (x_{2}) - \sigma z_{2}\bigr)  \notag\\
&+ \, q_2 \, \biggl(\left(1 - \sigma\right) x_{2}\, p \left({x_1}\sb\tau+x_2\right) - C_{2} (x_{2})- F \bigl( x_{2}\, p \left({x_1}\sb\tau+x_2\right) - z_{2} \bigr)\biggr).
\end{align}

\medskip
\noindent
The four-dimensional Cournot duopoly nonlinear model with tax evasion and time-delay  under consideration, follows a  gradient dynamic approach, that is  
\begin{equation}\label{tau_particular}
\begin{split}
\frac{dx_{ 1 }}{dt} \ ={ k }_{ 1 }\frac { \partial { P }_{ 1 }(\vec{x},\vec{x}\sb\tau) }{ \partial { x }_{ 1 } }=&{ k }_{ 1 }\biggl( \left[ 1-{ q }_{ 1 }\sigma -{ q }_{ 1 }{ F }^{ \prime  }({ x }_{ 1 }p({ x }_{ 1 }+{ x }_{ 2 })-{ z }_{ 1 } \right] \left[ p({ x }_{ 1 }+{ x }_{ 2 })+{ x }_{ 1 }{ p }^{ \prime  }({ x }_{ 1 }+{ x }_{ 2 }) \right] -\rv{C^{ \prime  }_{ 1 }}({ x }_{ 1 }) \biggr), \\
\frac{dx_{ 2 }}{dt}\ ={ k }_{ 2 }\frac { \partial { P }_{ 2 }(\vec{x},\vec{x}\sb\tau) }{ \partial { x }_{ 2 } } =&{ k }_{ 2 }\biggl( \left[ 1-{ q }_{ 2 }\sigma -{ q }_{ 2 }{ F }^{ \prime  }({ x }_{ 2 }p({ x }_{ 1 \tau}+{ x }_{ 2 })-{ z }_{ 2 } \right]\left[ p({ x }_{ 1 \tau }+{ x }_{ 2 })+{ x }_{ 2 }{p }^{ \prime  }({ x }_{ 1 }+{ x }_{ 2 }) \right] -\rv{ C^{ \prime  }_{ 2 }}({ x }_{ 2 }) \biggr), \\ 
\frac{dz_{ 1 }}{dt}\ ={ k }_{ 3 }\frac { \partial { P }_{ 1 }(\vec{x},\vec{x}\sb\tau) }{ \partial { z }_{ 1 } } =&{ k }_{ 3 }\left[ -(1-{ q }_{ 1 })\sigma +{ q }_{ 1 }{ F }^{ \prime  }({ x }_{ 1 }p({ x }_{ 1 }+{ x }_{ 2 })-{ z }_{ 1 } \right], \\ 
\frac{dz_{ 2 }}{dt}\ ={ k }_{ 4 }\frac { \partial { P }_{ 2 } (\vec{x},\vec{x}\sb\tau)}{ \partial { z }_{ 2 } } =&{ k }_{ 4 }\left[ -(1-{ q }_{ 2 })\sigma +{ q }_{ 2 }{ F }^{ \prime  }({ x }_{ 2 }p({ x }_{ 1 \tau}+{ x }_{ 2 })-{ z }_{ 2 } \right].
\end{split}
\end{equation}

\rv{
In the following proposition, we compute the equilibrium point  of the system~(\ref{tau_particular}). It is a slight generalization of \cite[Proposition~4]{neamctu2010deterministic}.
}%
\begin{proposition}\label{PropEquilibrium}
Consider the Cournot duopoly nonlinear model with tax evasion and time-delay (\ref{tau_particular}) with Assumption A. Then, the equilibrium point of system (\ref{tau_particular}) is
\begin{equation}\label{equilibrium}
x_1^*=\dfrac{1-\sigma}{(c_1+c_2)^2}c_2,\quad x_2^*=\dfrac{1-\sigma}{(c_1+c_2)^2}c_1,\quad
z_1^*=\dfrac{c_2}{c_1+c_2}-\dfrac{1-q_1}{sq_1},\quad z_2^*=\dfrac{c_1}{c_1+c_2}-\dfrac{1-q_2}{sq_2}.
\end{equation}
\end{proposition}

\begin{proof}
\rv{The equilibrium point of system (\ref{tau_particular}) satisfy that
\begin{align}
0=&\left[ 1-{ q }_{ l }\sigma -{ q }_{ l }{ F }^{ \prime  }(u) \right] \left[ p(y)+{ x }_{ 1 }{ p }^{ \prime  }(y) \right] -C^{ \prime  }_{ l }({ x }_{ 1 })\label{Ec1},\\
0=&-(1-{ q }_{ l })\sigma +{ q }_{ l }{ F }^{ \prime  }(u).\label{Ec2}
\end{align}
where $C_l(x_l)=c_lx_l$, $p(y)=1/y$, $F(u)=\frac{1}{2}s\sigma u^2$, $y=x_1+x_2$ and $u=x_{l}p(y)-z_{l}$, $l=1,2$. Note that, ${ x }_{l}={ x }_{ l \tau}$, $l=1,2$.  From (\ref{Ec2}) follows that ${ F }^{ \prime  }(u)= \frac{(1-{ q }_{ l })\sigma}{{ q }_{ l }}$, and 
substituting the above equation into (\ref{Ec1}) we get
\begin{align*}
 0&=\left(1-\sigma\right)\left(p(y)+x_{l}p'(y)\right)-C_{l}'(x_{l})=\left(1-\sigma\right)\left(\frac{1}{(x_1+x_2)}+\frac{x_{l}}{(x_1+x_2)^{2}}\right)-c_{l},
\end{align*}
therefore,
\begin{equation}
 x_l= (x_1+x_2)-\frac{c_l}{1-\sigma}(x_1+x_2)^2,  \quad l=1,2;\label{xl}
\end{equation}
On the other hand
\begin{align*}
 0=&-(1-{ q }_{ l })\sigma +{ q }_{ l }{ F }^{ \prime  }(u)=-(1-{ q }_{ l })\sigma +{ q }_{ l }s\sigma(x_{l}p(x_1+x_2)-z_{l}),
\end{align*}
hence
\begin{equation}
 z_l= \frac{x_l}{(x_1 + x_2)}-\frac{1-q_l}{q_l s}, \quad l=1,2.\label{zl}   
\end{equation}
From (\ref{xl}) we have
\begin{equation*}
    x_1=(x_1+x_2)-\frac{c_1}{1-\sigma}(x_1+x_2)^2\quad \text{and}\quad
    x_2=(x_1+x_2)-\frac{c_2}{1-\sigma}(x_1+x_2)^2,
\end{equation*}
adding the previous equations we have
$$x_1=\frac{1-\sigma}{c_1+c_2}-x_2.$$
Thus, 
\begin{align}
x_2=&(x_1+x_2)-\frac{c_2}{1-\sigma}(x_1+x_2)^2=\left(\frac{1-\sigma}{c_1+c_2}\right)-\frac{c_2}{1-\sigma}\left(\frac{1-\sigma}{c_1+c_2}\right)^2=\frac{1-\sigma}{(c_1+c_2)^2}c_1,\label{x1}\\
x_1=&\frac{1-\sigma}{c_1+c_2}-x_2=\frac{1-\sigma}{c_1+c_2}-\frac{1-\sigma}{(c_1+c_2)^2}c_1=\frac{1-\sigma}{(c_1+c_2)^2}c_2.\label{x2}
\end{align}
The points $z_l$, $l=1,2$, are obtained by replacing (\ref{x1}) and (\ref{x2}) in (\ref{zl}).}
\end{proof}

\medskip
\noindent
Next, some results regarding to delay-dependent and delay-independent stability are presented.

\subsection{Delay-dependent and Delay-independent stability}

Now, a stability analysis of Cournot duopoly model is performed when the Assumption A is considered and the marginal cost rate $\mu=\frac{c_2}{c_1}$ is introduced. Substituting $c_2=\mu c_1$, we may reformulate Assumption A as:

\medskip
\noindent
\textbf{Assumption B:} Let $p(y)=1/y$, $F(y)=\frac{1}{2}s\sigma y^2$,  $C_1(x_1)=c_1\,x_1$ and $C_2(x_2)=\mu\,c_1\,x_2$, where $s\geq 1$, $c_1>0$,  are constants, $\mu=\frac{c_2}{c_1}$ and $y=x_1+x_2$.

\medskip
\noindent
Therefore Proposition~\ref{PropEquilibrium} can be reformulated as:

\begin{proposition}\label{Equilibrium_DelayDep}
Let Cournot duopoly model with time-delay (\ref{tau}) such that Assumption B is met, then the equilibrium point of above model is
\rv{\begin{equation} 
x_{1}^{*}= {\frac {\mu\, \left(1- \sigma \right) }{{\it c_1}\, \left( 
1+\mu \right) ^{2}}},\quad x_{2}^{*} = \dfrac{1-\sigma}{c_1(1 +\mu)^{2}},\quad z_{1}^{*}= \dfrac{\mu}{1+\mu} - \dfrac{1-q_1}{s\,q_1},\quad z_{2}^{*} = \dfrac{1}{1+\mu} - \dfrac{1-q_1}{s\,q_1}.
\end{equation}}
 \end{proposition}
\begin{proof}
Follow from Proposition~\ref{PropEquilibrium}.
\end{proof}

\medskip
\noindent
Below, the main result of this paper is stated and proved.
\begin{theorem}\label{Corollary_DelayDep}
Let Cournot duopoly linear model with time-delay (\ref{GralLineal_duopolio}) be
stable for $\tau=0$ and Assumptions  B {is} satisfied.
Then, the system (\ref{GralLineal_duopolio}) is delay-dependent stable if $\mu\in(0,3-2\,\sqrt{2})\cup(3+2\,\sqrt{2},\infty)$. Moreover, the Cournot duopoly nonlinear model with time-delay (\ref{tau}) have a bifurcation occurs at 
\begin{equation*}
\tau_{0}=\frac { 1 }{ { \omega_0  } } { tan }^{ -1 }\left(\frac { N_2(\omega_0) }{ N_1(\omega_0) } \right).
\end{equation*}

On the other hand, the system (\ref{GralLineal_duopolio}) is delay-independent stable if $\mu\in[3-2\,\sqrt{2},\ 3+2\,\sqrt{2}]$.

\medskip
\noindent
Here, $N_1(\omega_0)= -\Phi(\omega_0)\left( ac{ \omega_0  }^{ 2 }-bd \right) -\Theta(\omega_0)\left( ad+bc \right)\omega_0$,
$N_2(\omega_0)=-\Phi(\omega_0)\left( ad+bc \right)\omega_0 +\Theta(\omega_0) \left( ac{ \omega_0  }^{ 2 }-bd \right)$, with $\Phi(\omega_0)={\omega}_0^{4}- \left(\alpha_{{1}}\beta_{{1}}+\alpha_{{0}}+\beta_{{0}} \right) {\omega}_0^{2}+\alpha_{{0}}\beta_{{0}}$, $\Theta(\omega_0)={\omega}_0^{3} \left( \alpha_{{1}}+\beta_{{1}} \right) - \left( \alpha_{{0}}\beta_{{1}}+\alpha_{{1}}\beta_{{0}} \right) \omega_0$; and
\begin{align*}
\alpha_{0} =& {\frac {2{\it k_1}\,{{\it c_1}}^{2}s{\it q_1}\,{\it k_3}\,\sigma\, \left( \mu+1 \right) }{1-\sigma}},\\
\alpha_{1} =& {\frac {-s{\it k_3}\,{\it q_1}\,{\sigma}^{2} \left( \sigma-2 \right) -
 \left( {\it k_1}\, \left( {\it q_1}\,s-2(\mu+1) \right) {{\it c_1}}^{2}
+s{\it q_1}\,{\it k_3} \right) \sigma-2\,{{\it c_1}}^{2}{\it k_1}\,
 \left( \mu+1 \right) }{ \left( -1+\sigma \right) ^{2}}},\\
\beta_{0} =& {\frac {2{\it k_2}\,{{\it c_1}}^{2}\mu\,s{\it q_2}\,{\it k_4}\,\sigma\, \left( \mu+1 \right) }{1-\sigma}}, \\
\beta_{1} =& {\frac { \left( -{\it k_2}\,\mu\, \left( \mu\,{\it q_2}\,s
-2(\mu+1) \right) {{\it c_1}}^{2}-s{\it q_2}\,{\it k_4} \right) \sigma-2\,{{\it c_1
}}^{2}{\it k_2}\,\mu\, \left( \mu+1 \right) -s{\it k_4}\,{\it q_2}\,{
\sigma}^{2} \left( \sigma-2 \right) }{ \left( -1+\sigma \right) ^{2}}}, \\
a =& {\frac {{\it k_1}\, \left(  \left( \mu\,{\it q_1}\,s-{\mu}^{2}+1 \right) \sigma+{\mu}^{2}-1 \right) {{\it c_1}}^{2}}{ \left( -1+\sigma \right) ^{2}}},\quad
b ={\frac {{\it k_1}\,{{\it c_1}}^{2}s{\it q_1}\,{\it k_3}\,\sigma\, \left( \mu-1 \right)  \left( \mu+1 \right) }{-1+\sigma}}, \\
c =& {\frac {{\it k_2}\, \left(  \left( \mu\,{\it q_2}\,s+{\mu}^{2}-1 \right) \sigma-{\mu}^{2}+1 \right) {{\it c_1}}^{2}}{ \left( -1+\sigma \right) ^{2}}},\quad
d ={\frac {{\it k_2}\,{{\it c_1}}^{2}s{\it q_2}\,{\it k_4}\,\sigma\, \left( \mu-1 \right)  \left( \mu+1 \right) }{-1+\sigma}}.
\end{align*}
\end{theorem}
\begin{proof}
\rv{Note that the existence of a (bifurcation) critical parameter $\tau_0$ directly depends on the existence of a positive root $\omega_0\neq0$ such that the polynomial given in (\ref{Polynomial_omega}) satisfy $p(\omega_0)=0$. Hence, if this $\omega_0$ does not exist, then there is not $\tau_0$. Thus, delay-dependent stability  is reduced to obtain conditions for the existence of $\omega_0$, below are some arguments in this regard.} 

\medskip
\noindent
It is well-known that a polynomial of the form $P(\omega)=a_n\omega^n+a_{n-1}\omega^{n-1}+\ldots+a_1\omega+a_0$ has at least a positive root $\omega_0\neq0$, if $a_n$ is positive and $a_0$ is negative. Based on the foregoing and observing that the coefficient of the polynomial (\ref{Polynomial_omega}) are
\begin{align*}
a_{12}&={\frac {{{\it k_1}}^{2}{{\it c_1}}^{8} \left( -\mu\,{\it q_1}\,s\sigma+{
\mu}^{2}\sigma-{\mu}^{2}-\sigma+1 \right) ^{2}{{\it k_2}}^{2} \left( 
\mu\,{\it q_2}\,s\sigma+{\mu}^{2}\sigma-{\mu}^{2}-\sigma+1 \right) ^{2}
}{ \left( -1+\sigma \right) ^{8}}},\\
a_{0}&=-{\frac {{{\it k_1}}^{4}{{\it c_1}}^{16}{s}^{8}{{\it q_1}}^{4}{{\it k_3}}^
{4}{\sigma}^{8} \left( \mu-1 \right) ^{4} \left( \mu+1 \right) ^{10}{{
\it k_2}}^{4}{{\it q_2}}^{4}{{\it k_4}}^{4} \left( {\mu}^{2}-6\,\mu+1
 \right) }{ \left( -1+\sigma \right) ^{8}}},
\end{align*}
\rv{where $a_{12}$ is always positive, while  $a_0$ is negative for $\mu\in(0,3-2\,\sqrt{2})\cup(3+2\,\sqrt{2},\infty)$. We can conclude that}
at least there is an \rv{$\omega_0>0$} solution of the polynomial (\ref{Polynomial_omega}) and using Theorem~\ref{theorem_gral}, the first part of result follows.

\medskip
\noindent
On the other hand, \rv{a polynomial of the form} $P(\omega)=a_n\omega^n+a_{n-1}\omega^{n-1}+\ldots+a_1\omega+a_0$ has no positive root  $\omega_0\neq0$, if all its $a_l$ are positive.
As mentioned earlier,  $a_{12}$ is always positive and $a_0$ is positive if $\mu\in[3-2\,\sqrt{2},\ 3+2\,\sqrt{2}]$. However, the coefficients $a_l$, $l=10,8,6,4,2$, of the polynomial (\ref{Polynomial_omega}) are extensive and the analytically demonstration of its positivity is not trivial. Using Lagrange multipliers and numerical methods,  the minimums of the coefficients $a_l$ are shown to be zero, as depicted in Table~\ref{TableMinimun}.

\begin{table}
\caption{Minimum values of the coefficients $a_{l}$ of (\ref{Polynomial_omega}).}\label{TableMinimun}
\begin{tabular*}{.8\linewidth}{@{\extracolsep{\fill}}cccccc}\hline
 & $a_{10}$ & $a_{8}$   & $a_{6}$  & $a_{4}$   & $a_{2}$     \\  \hline
min      &   0      &   0       &  0       &   0       &   0         \\ 
$\mu $   &   0.17   &   3.72    &  0.48    &   2.99    &   3         \\ 
$c_1$    &   0      &   43693   &  31889.2 &   49734.3 &   50000     \\ 
$k_{1}$  &   0      &   0       &  0       &   0       &   50000     \\ 
$k_2$    &   0      &   72553.2 &  10947.7 &   49388.4 &   50000     \\ 
$k_3$    &   50000  &   88383.8 &  13228.7 &   48492.4 &   50000     \\ 
$k_4$    &   50000  &   36467.5 &  0       &   48746.1 &   50000     \\ 
$q_1$    &   0      &   0.68    &  0.13    &   0.49    &   0.50      \\ 
$q_2$    &   0      &   0.33    &  0.15    &   0.49    &   0.50      \\ 
$s$      &   1      &   49881.6 &  10915.3 &   49318.5 &   50000   \\ 
$\sigma$ &   0      &   0.51    &  0.16    &   0.33    &   0         \\  \hline
\end{tabular*}
\end{table}

\smallskip
\noindent
Thus, using Theorem~\ref{theorem_gral}, the second part of result follows.  
\end{proof}

\begin{remark}
It is puzzling that the interval of delay-independent stability in Theorem~\ref{Corollary_DelayDep} is precisely the interval for the stability  for a Cournot duopoly model presented in \cite{A1}. 
\end{remark}

The particular case when $\mu=1$ will further be analyzed next. It corresponds to the case when both firms have the same marginal costs. Although the following assertions are direct consequence of our previous results, we believe it to be interesting that a simplified hypothesis allow direct computations of the coefficients of the polynomial~(\ref{Polynomial_omega}). In this case, for instance, $a_0$ and $a_2$ are both zero.

%

%

\medskip
\noindent
\textbf{Assumption C:} the probability of being audited and detected of both firms is the same, $q_2=q_1$, the constants $k_j$, $j=2,3,4$ are equals to $k_1$ and $\mu=1$.

\begin{corollary}\label{Equilibrium_DelayIndep}
Consider the Cournot duopoly nonlinear model with time-delay (\ref{tau}) with
Assumptions  B and C. The equilibrium point is
\begin{equation*}
x_{1}^{*} = \dfrac{1-\sigma}{4c_1} =  x_{2}^{*}, \hspace{1.5cm} z_{1}^{*} = \dfrac{1}{2} -\dfrac{1-q_1}{q_1 s} = z_{2}^{*}.
\end{equation*}
\end{corollary}
\begin{proof}
Follow from Proposition~\ref{Equilibrium_DelayDep}.
\end{proof}

\begin{corollary}\label{Corollary_DelayIndep}
Consider the Cournot duopoly linear model with time-delay (\ref{GralLineal_duopolio})
with Assumptions B and C. Then the Cournot duopoly linear model is delay-independent stable. In other words, if the quasi-polynomial is stable for $\tau=0$, then the quasi-polynomial will remain stable for all $\tau>0$.
\end{corollary}
\begin{proof}
Note that the constants $a_{12}$, $a_{10}$, \ldots, $a_{0}$ of the polynomial (\ref{Polynomial_omega}) are 
\begin{align*}
a_{12} =&{\frac {{{\it q_1}}^{4}{s}^{4}{\sigma}^{4}{{\it k_1}}^{4}{{\it c_1}}^{8}}{ \left( -1+\sigma \right) ^{8}}}, \\
a_{10}=& \left( 2\,{\frac {{{\it q}}_1^{4}{\sigma}^{4}{s}^{4} \left( {\it q_1}\,
s\sigma-4\,\sigma+4 \right) ^{2}{{\it c}_1}^{12}}{ \left( -1+\sigma
 \right) ^{12}}}+4\,{\frac {{s}^{6}{\sigma}^{6}{{\it q}}_1^{6}{{\it c}
}_1^{10}}{ \left( -1+\sigma \right) ^{10}}}+2\,{\frac {{s}^{6}{\sigma}^{
6}{{\it q1}}^{6}{{\it c}}_1^{8}}{ \left( -1+\sigma \right) ^{8}}}
 \right) {{\it k}}_1^{6},\\
a_{8} =& \Biggl[  \left( {\frac {16\,{s}^{7}{\sigma}^{7}{{\it q}}_1^{7}}{
 \left( 1-\sigma \right) ^{15}}}+{\frac {96\,{s}^{6}{\sigma}^{6}{{
\it q}}_1^{6}}{ \left( -1+\sigma \right) ^{14}}}+{\frac {256\,{{\it q}
}_1^{5}{s}^{5}{\sigma}^{5}}{ \left(1-\sigma \right) ^{13}}}+{
\frac {256\,{{\it q}}_1^{4}{\sigma}^{4}{s}^{4}}{ \left( -1+\sigma \right) ^{
12}}} \right) {{\it c}}_1^{16}+{\frac {4\,{\sigma}^{6}{s}^{6}{{\it q}}_1
^{6} \left( {\it q_1}\,s\sigma-4\,\sigma+4 \right) ^{2}
}{ \left( -1+\sigma \right) ^{14}}}{{\it c}}_1^{14}\\
&+ \left( {\frac {6\,{s}^{8}{\sigma}
^{8}{{\it q}}_1^{8}}{ \left( -1+\sigma \right) ^{12}}}+{\frac {16\,{s}^
{7}{\sigma}^{7}{{\it q}}_1^{7}}{ \left(1-\sigma \right)^{11}}}+{
\frac {64\,{s}^{6}{\sigma}^{6}{{\it q}}_1^{6}}{ \left( -1+\sigma \right) ^{
10}}} \right) {{\it c}}_1^{12}+{\frac {4\,{s}^{8}{\sigma}^{8}{{\it q}}_1
^{8}}{ \left( -1+\sigma \right) ^{10}}}{{\it c}}_1^{10}+{\frac {{s}^{8}
{\sigma}^{8}{{\it q}}_1^{8}}{ \left( -1+\sigma \right) ^{
8}}}{{\it c}}_1^{8} \Biggr] {{\it k}}_1^{8},\\
a_{6} =& {\frac {32\,{{\it q}}_1^{6}{\sigma}^{6}{s}^{6} \left( 4+ \left( {\it q_1
}\,s-4 \right) \sigma \right) ^{2}{{\it k}}_1^{10}}{
 \left( -1+\sigma \right) ^{14}}}{{\it c}}_1^{16}+{\frac {64\,{s}^{8}{
\sigma}^{8}{{\it q}}_1^{8}{{\it k}}_1^{10}}{ \left( -1+\sigma \right) ^{
12}}}{{\it c}}_1^{14}+{\frac {32\,{s}^{8}{\sigma}^{8}{{\it q}}_1^{8}{{\it k}}_1^{10}}{ \left( -1+\sigma \right) ^{10}}}{{
\it c}}_1^{12},\\
a_{4}=&{\frac {256\,{{\it q}}_1^{8}{s}^{8}{
\sigma}^{8}{{\it k}}_1^{12}}{ \left( -1+\sigma \right) ^{12}}}{{\it c}}_1^{16}\\
a_2=&a_0=0.
\end{align*}
Since $a_{12},\ldots,a_4$ are positive and both  $a_2$ and $a_0$ are zero, \rv{we conclude that the polynomial (\ref{Polynomial_omega}) has no positive real roots $\omega_0\neq0$ and the result follows.}
\end{proof}

\section{Simulation of Results}
Below, some numerical simulations are presented to illustrate the theoretical results obtained in the previous section using Matlab's Simulink.

\medskip
Without loss of generality, consider the Cournot duopoly nonlinear model with tax evasion and time-delay given in (\ref{tau_particular}), with  $\sigma=0.1$, $s =40$, $q_{1,2}=0.5$, $k_{1,2,3,4}=1$ and  $c_2:=\mu c_1$, $c_1=0.1$:
\begin{equation}\label{NLS_simulation}
\begin{split}
\frac{dx_{ 1 }}{dt}=& \left(\frac{1}{\left({\it x_1}+{\it x_2}\right)} - \frac {{\it x_1}}{ \left( {\it x_1}+{\it x_2} \right) ^{2}} \right)\left(0.95- 2\,
 \left( {\frac{{\it x_1}}{\left({\it x_1}+{\it x_2}\right)}}-{\it z_1} \right) \right)- {0.1},\\ 
\frac{dx_{ 2 }}{dt}=&\left(\frac{1}{\left( {x_1}\sb\tau +
{\it x_2} \right)}-\frac{{\it x_2}}{ \left( {x_1}\sb\tau  +{
\it x_2} \right) ^{2}}   \right) \left(0.95-2\, \left( {\frac {{\it x_2}
}{\left({x_1}\sb\tau+{\it x_2}\right)}}-{\it z_2} \right) \right)- \mu\,{0.1},\\
\frac{dz_{ 1 }}{dt}=&\,2\left( \frac {{\it x_1}}{{\it x_1}+{\it x_2}}-{\it z_1}\right)- 0.05,\\
\frac{dz_{ 2 }}{dt}=&\,2\left( \frac {{\it x_2}}{{x_1}\sb\tau+{\it x_2}}-{\it z_2}\right)- 0.05.
\end{split}
\end{equation}
Thus, using Proposition~\ref{Equilibrium_DelayDep} the equilibrium point of above system is  
\begin{equation} \label{equilibrium_DelayIndep}
x_{1}^{*} = \dfrac{0.09\mu}{(0.1 + 0.1\mu)^{2}}, \quad  x_{2}^{*} = \dfrac{0.9}{(0.1 +0.1\mu)^{2}},\quad z_{1}^{*} = \dfrac{0.1\mu}{0.1 + 0.1\mu} - 0.025, \quad  z_{2}^{*} = \dfrac{0.1}{0.1+0.1\mu} -0.025. 
\end{equation}

Cournot duopoly linear model with time-delay (\ref{GralLineal_duopolio}) is
\begin{equation*}
\dfrac{d \vec {x}(t)}{dt}=A\vec { x } +B\vec { x }_{\tau},\\
\end{equation*}
where, $A=\left[ \begin {array}{cccc} - 0.02\,\mu- 0.046&
 0.01\,({\mu}^{2}-1)+ 0.024\,\mu&
 0.4\,\mu+ 0.2&0\\ \noalign{\medskip}0& \left( -
 0.02- 0.046\,\mu \right) \mu&0& 0.2\,\mu
\\ \noalign{\medskip} 0.2&- 0.2\,\mu&- 2&0
\\ \noalign{\medskip}0& 0.2\,\mu&0&- 2\end {array} \right]$, and\newline $B=\left[ \begin {array}{cccc} 0&0&0&0\\ \noalign{\medskip}-
 0.01\,{\mu}^{2}+ 0.024\,\mu+ 0.01&0&0&0
\\ \noalign{\medskip}0&0&0&0\\ \noalign{\medskip}- 0.2&0&0&0
\end {array} \right]$.

\medskip
\noindent
The quasi-polynomial (\ref{quasi}) is
\begin{equation}\label{Q_simulation}
q(\lambda,\tau)=p_{1}(\lambda)p_{2}(\lambda)-e^{-\lambda \tau}(a\lambda-b)(c\lambda-d),
\end{equation}
with $\alpha_1=-2.04-0.02 \mu$, $\beta_1= -0.12 \times 10^{-2}\mu (18\mu-2)-2-0.024 \mu (\mu+1)$,
$\alpha_0=0.04\mu +0.04$, $\beta_0=0.04\mu (\mu+1)$,
$a= 0.01\mu^{2}+0.024\mu -0.01$, $b=-0.02(\mu-1)(\mu+1)$,
$c=-0.01\mu^{2}+0.024\mu+0.01$, $d=0.02(\mu-1)(\mu+1)$.
The polynomial  $P(\omega)$ given in (\ref{Polynomial_omega}) is
\begin{align*}
P(\omega)&=a_{12}\omega^{12} + a_{10}\omega^{10} + a_{8}\omega^{8}+ a_{6}\omega^{6}+a_{4}\omega^{4}+a_{2}\omega^{2} +a_{0}=0.
\end{align*}
Here, 
\begin{equation*}
\begin{split}
a_{12}=&\, 0.23\!\times\!10^{-7}\left(  0.9\,{\mu}^{2}+ 2.0\,\mu- 0.9 \right) ^
{2} \left( - 0.9{\mu}^{2}+ 2.0\mu+ 0.9 \right) ^{2},\\
a_{10}=&\, { 3.35\!\times\!10^{-11}}{\mu}^{12}+
{ 3.17\!\times\!10^{-11}}{\mu}^{11}+
 0.105\!\times\!10^{-8}{\mu}^{10} -0.409\!\times\!10^{-9}{\mu}^{9}+ 0.225\!\times\!10^{-6}
{\mu}^{8}\\ 
&\qquad + 0.115\!\times\!10^{-7}{\mu}^{7}
-0.152\!\times\!10^{-5}{\mu}^{6}+ 0.115\!\times\!10^{-8}{
\mu}^{5}+ 0.56\!\times\!10^{-5}{\mu}^{4}-
 0.409\!\times\!10^{-9}{\mu}^{3}\\
&\qquad - 0.159\!\times\!10^{-5}
{\mu}^{2}+{ 3.177\!\times\!10^{-11}}\mu +0.24\!\times\!10^{-6}, 
\\
\vdots =&\quad \vdots\\
a_{2} =&-{5.94\!\times\!10^{-14}}{\mu}^{16
}+{2.2\!\times\!10^{-26}}{\mu}^{15}+{
 1.95\!\times\!10^{-12}}{\mu}^{14}+{ 6.5
\!\times\!10^{-12}}{\mu}^{13}+ 2.02\!\times\!10^{-9}{\mu}^{12}\\
&\qquad + 4.03\!\times\!10^{-9}{\mu}^{11}
-4.15\!\times\!10^{-9}{\mu}^{10}-1.23\!\times\!10^{-8}
{\mu}^{9}-1.71\!\times\!10^{-9}{\mu}^{8}+
8.85\!\times\!10^{-9}{\mu}^{7}\\ %
&\qquad + 7.88\!\times\!10^{-9}
{\mu}^{6}+ 6.93\!\times\!10^{-9}{\mu}^{5}
-2.20\!\times\!10^{-9}{\mu}^{4}-1.13\!\times\!10^{-8}
{\mu}^{3}-3.75\!\times\!10^{-9}{\mu}^{2}\\ %
&\qquad + 3.85\!\times\!10^{-9}\mu-1.92\times 10^{-10},\\
a_{0} =&-{ 5.94\times 10^{-14}}\, \left( \mu-1 \right) ^{4} \left( \mu+
1 \right) ^{10} \left( {\mu}^{2}-6\,\mu+1 \right). 
\end{split}
\end{equation*}

Immediately, some simulations of the Cournot duopoly nonlinear model are presented using Simulink-Matlab for some values of $\mu$.

\subsection{Delay-independent stability}

Using \rv{Theorem}~\ref{Corollary_DelayDep} for $\mu\in[3-2\,\sqrt{2},\ 3+2\,\sqrt{2}]=[0.1716,5.8284]$
the Cournot duopoly nonlinear model with tax evasion and time-delay (\ref{NLS_simulation}) is delay-independent stability. In Table~\ref{m_indep}, some equilibrium points are obtained for different values of $\mu$. 
\rv{In Figure~\ref{FigDelayIndep}, three maps of the roots location of the quasi-polynomial (\ref{Q_simulation}) in the complex plane when $\mu\in[0.1716,5.8284]$ are presented. In the first map, $\mu=0.1716$ is fixed and $\tau=0,\ 10,\ 100,\ 1000$ varying. In the second, 
$\mu=3.8$ and $\tau=0,\ 10,\ 100,\ 1000$. While, $\mu = 5.8284$ and $\tau=0,\ 10,\ 100,\ 1000$ in the third. In the three maps, it can be seen that the roots approximate the imaginary axis when $\tau$ increases, even more these roots appear to form an abscissa close to the imaginary axis, but never touch this axis. It should be noted that, increasing the value of $\tau$, implies increasing the difficulty in graphing the roots due to the increase in computational calculation to approximate the roots. Note that, all roots of (\ref{Q_simulation}) remain in the open left-half of complex plane.
This behavior is similar for all $\tau>0$ and $\mu\in[3-2\,\sqrt{2},\ 3+2\,\sqrt{2}]$. This exemplifies the postulated in the theoretical results above. The roots can be calculated using QPmR \cite{vyhlidal2009mapping}.}

\begin{table}
\caption{Equilibrium points of Cournot duopoly model (\ref{NLS_simulation}) when $\mu\in[0.1716,5.8284]$.}\label{m_indep}
\begin{tabular*}{.7\linewidth}{@{\extracolsep{\fill}}cccccc}\hline
$\mu$ & $x^{*}_{1}$ & $x^{*}_{2}$ & $z^{*}_{1}$ & $z^{*}_{2}$ \\ \hline
0.1716 & 1.12 & 6.5& 0.12 & 0.82\\ 
0.5    & 2 & 4 & 0.30 & 0.64\\ 
1    & 2.25 & 2.25 & 0.47 & 0.47\\ 
1.5   & 2.16 & 1.44 & 0.57 & 0.37\\ 
2     & 2    & 1    & 0.64 & 0.30\\ 
2.5   & 1.18 & 0.73 & 0.68 & 0.26\\ 
3     & 1.68 & 0.56 & 0.72 & 0.22\\ 
3.5   & 1.55 & 0.44 & 0.75 & 0.19\\ 
4     & 1.44 & 0.36 & 0.77 & 0.17\\ 
4.5   & 1.33 & 0.29 & 0.79 & 0.15 \\ 
5     & 1.25 & 0.25 & 0.80 & 0.14\\
5.8284& 1.12 & 0.19 & 0.82 & 0.12\\ \hline 
\end{tabular*}
\end{table}

\begin{figure}
 \centering
 \includegraphics[scale=.48]{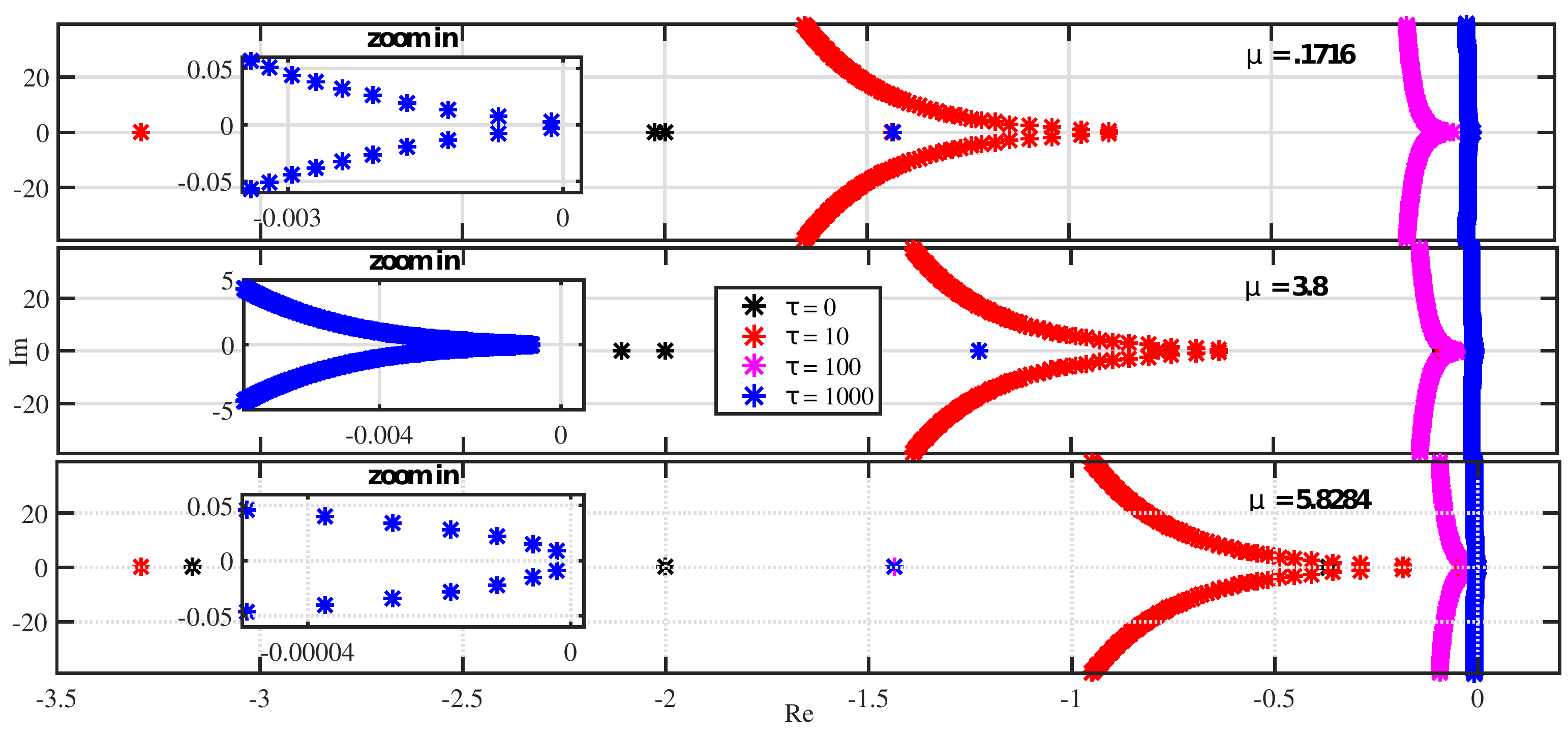}
 \caption{Delay-independent: roots location of the quasi-polynomial (\ref{Q_simulation}) in the complex plane.}\label{FigDelayIndep}
\end{figure}

\begin{figure}
 \centering
 \includegraphics[scale=.43]{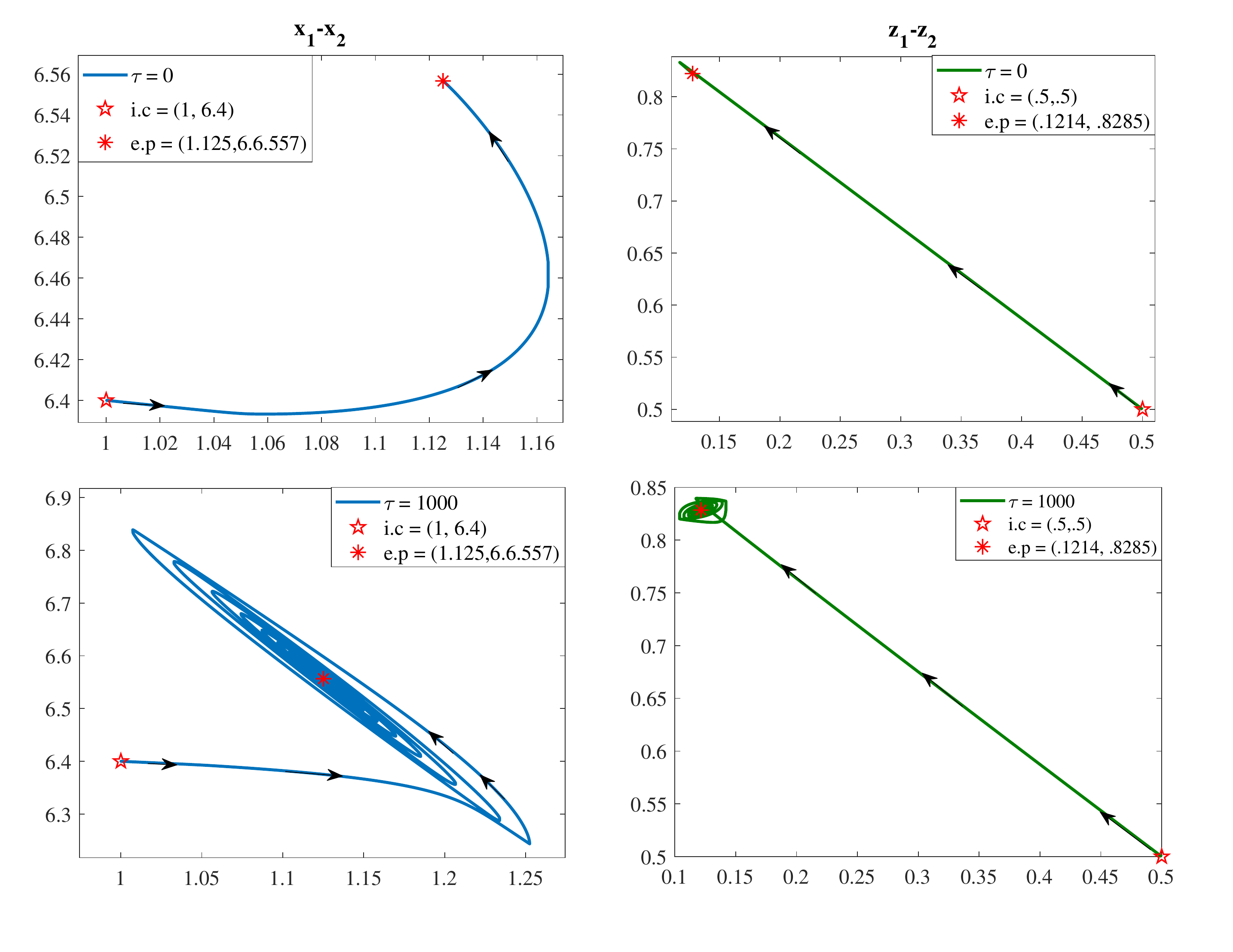}
 \caption{Phase diagram when $\mu=0.1716$.  }\label{fig-Phase3}
\end{figure}

\subsection{Delay-dependent stability}

Using \rv{Theorem}~\ref{Corollary_DelayDep} for $\mu\in(0,3-2\,\sqrt{2})\cup(3+2\,\sqrt{2},\infty)$ the Cournot duopoly model with tax evasion and time-delay (\ref{NLS_simulation}) is delay-dependent stability. Table \ref{m_dep} gives critical values of $\tau^*$ which can produce bifurcations and limit cycles in the system (\ref{NLS_simulation}). For illustrative purposes, in Figures~\ref{RootsDelayDepend}-\ref{fig-Phase2} a particular case of these parameters is shown.

\begin{table}
\caption{Equilibrium points and critical values $\tau^{*}$  of Cournot duopoly model (\ref{NLS_simulation}) when $\mu\in[0,3-2\,\sqrt{2})\cup(3+2\,\sqrt{2},\infty)$.}\label{m_dep}

\begin{tabular*}{.7\linewidth}{@{\extracolsep{\fill}}cccccc}\hline
$\mu$  & $x^{*}_{1}$ & $x^{*}_{2}$ & $z^{*}_{1}$ & $z^{*}_{2}$ & $\tau^{*}$     \\  \hline
0      & 0    & 9    & -0.2  & 0.95 &   246.4898206  \\ 
0.01   & 0.08 & 8.82 & -0.01 & 0.96 &  257.6637089  \\ 
0.04   & 0.33 & 8.32 & 0.1   & 0.93 &  299.1040366  \\ 
0.08   & 0.61 & 7.71 & 0.04  & 0.90 &  385.7162877  \\ 
0.1    & 0.74 & 7.43 & 0.06  & 0.88 &  455.8218422  \\ 
0.14   & 0.96 & 6.92 & 0.09  & 0.85 &  770.9037092  \\ 
6      & 1.11 & 0.18 & 0.83  & 0.11 &  64.72944712  \\ 
10     & 0.7438 & 0.07438 & 0.8841  & 0.0659 &  4.809451548  \\ 
100    & 0.08 & $0.88\times 10^{-3}$ & 0.96 & -0.01 &  0.060011892  \\
1000   & $0.88\times 10^{-2}$ & $0.89\times 10^{-5}$ & 0.97 & -0.02&   0.000579951  \\ \hline 
\end{tabular*}
\end{table}

\medskip
\noindent
\rv{In Figure \ref{RootsDelayDepend}, four maps of the roots location of the quasi-polynomial (\ref{Q_simulation}) in the complex plane are depicted when $\mu=10$ is fixed and $\tau=0,\ 3,\ 4.8,\ 5$ varying.
In the first map, if $\tau=0$ then (\ref{Q_simulation}) is a fourth-order polynomial, so it only has four roots in the open left-half of complex plane. While, in the second map  if $\tau\in(0,\tau^*)$ then the quasi-polynomial (\ref{Q_simulation}) has now an infinite number of roots, but all located in the open left-half of complex plane. Finally, when $\tau=\tau^*$ (\ref{Q_simulation}) has two dominant roots on the imaginary axis and when $\tau>\tau^*$ some roots cross to the right side causing the quasi-polynomial to be unstable. Therefore, the postulated in the theoretical results above is illustrated.} 

\begin{figure}
 \centering
 \includegraphics[scale=.47]{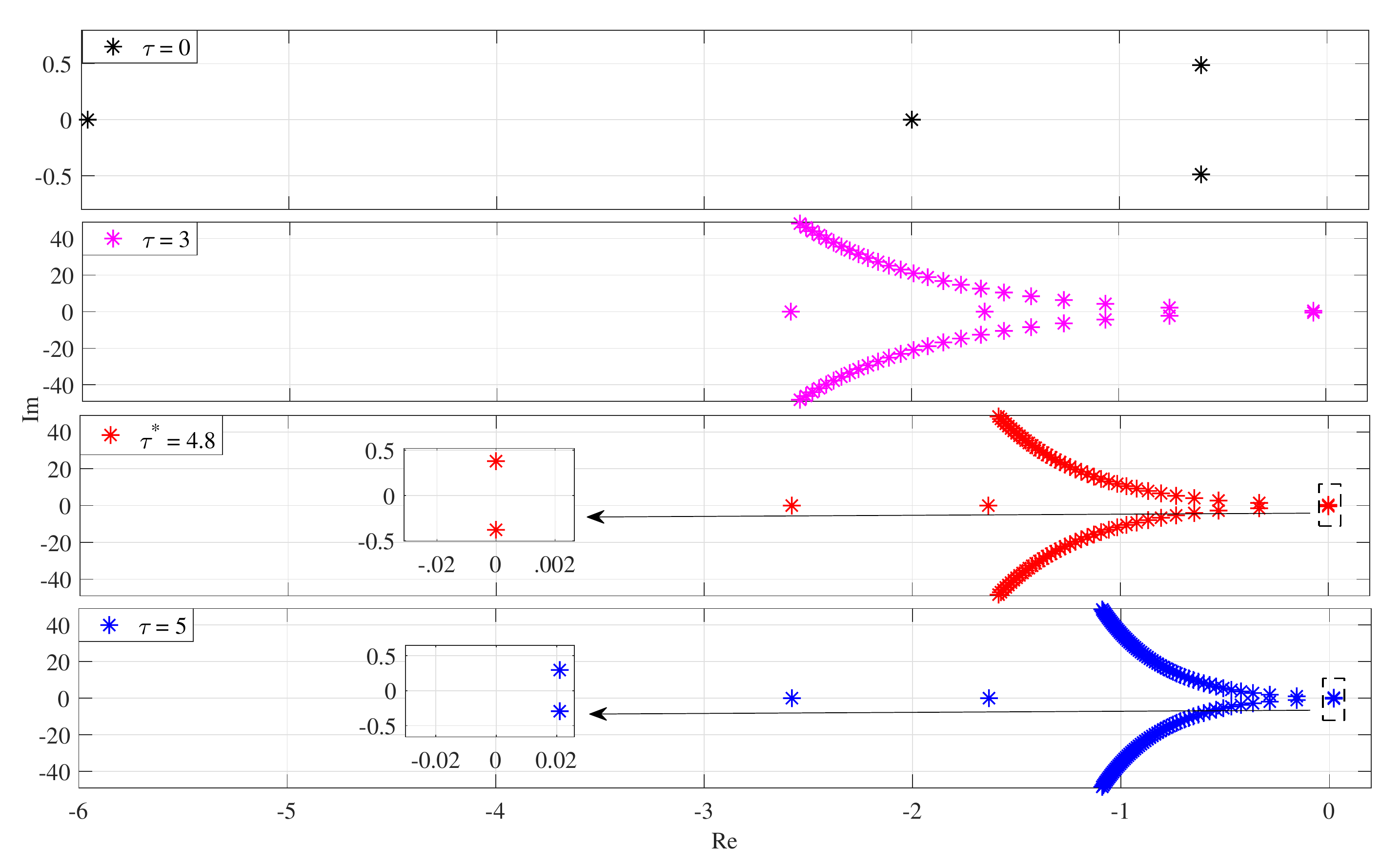}
 \caption{Delay-dependent: roots location of the quasi-polynomial (\ref{Q_simulation}) in the complex plane when $\mu=10$. }\label{RootsDelayDepend}
\end{figure}


\medskip
\noindent
On the other hand, the phase diagrams of the state variables in pairs $x_1$-$x_2$ and $z_1$-$z_2$ of Cournot duopoly nonlinear model with tax evasion and time-delay (\ref{NLS_simulation}) are presented in Figures~\ref{fig-Phase1} and \ref{fig-Phase2}. 
\begin{figure}
 \centering
 \includegraphics[scale=.43]{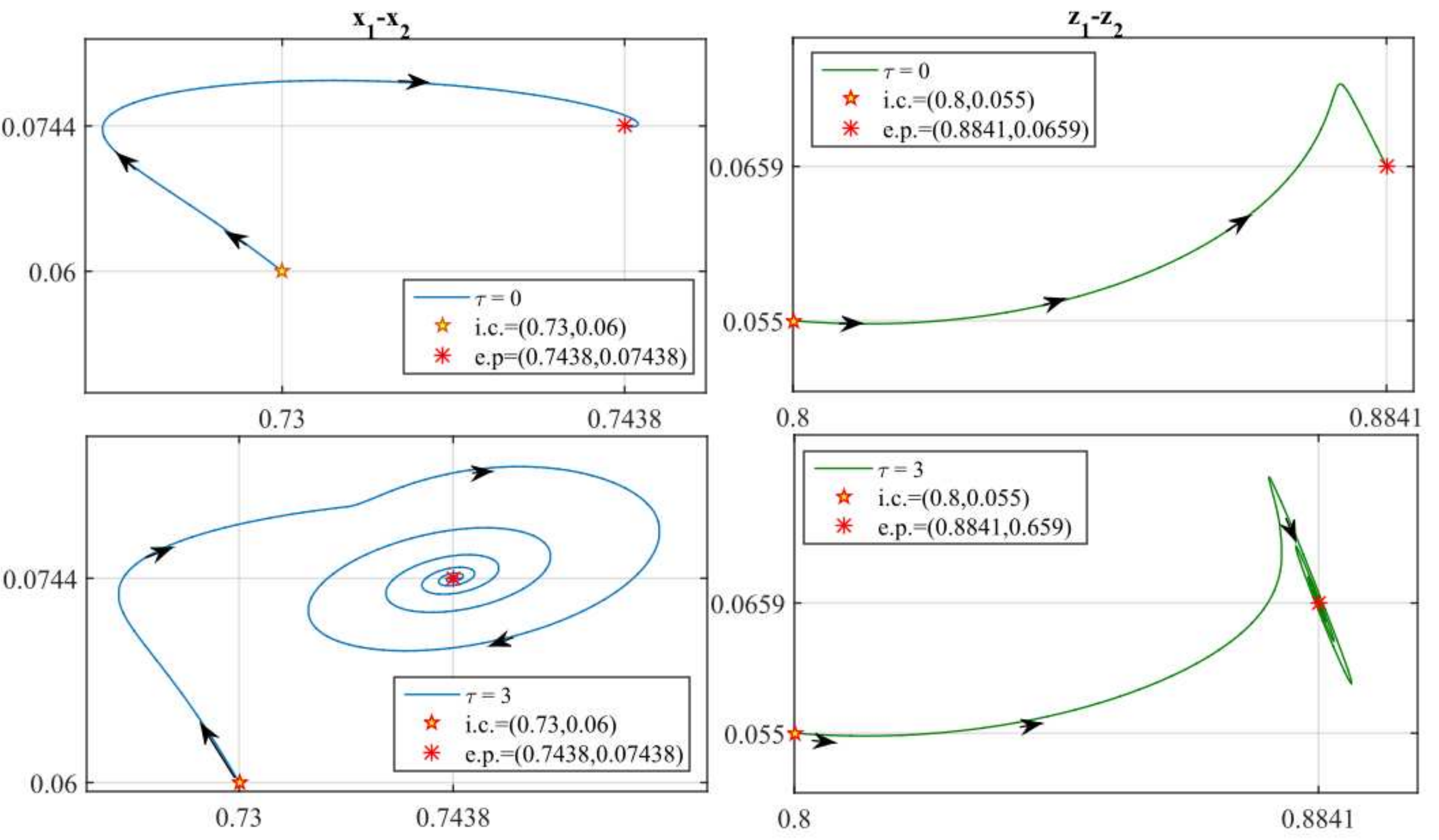}
 \caption{Phase diagram when $\mu = 10$ and $\tau=0$, $\tau=3$. }\label{fig-Phase1}
\end{figure}

\begin{figure}
 \centering
 \includegraphics[scale=.43]{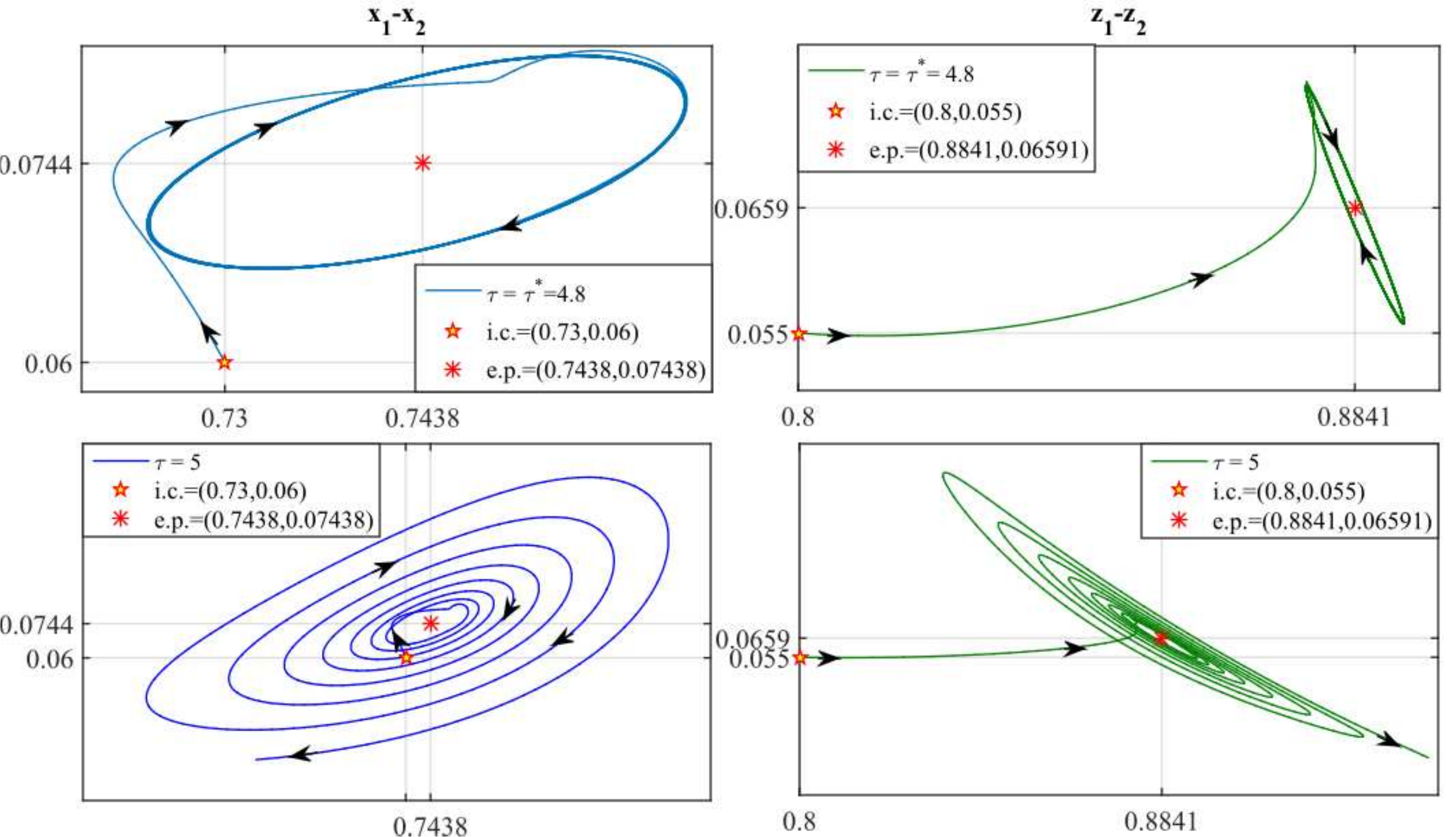}
 \caption{Phase diagram when $\mu = 10$ and $\tau=\tau^*$, $\tau=5$. }\label{fig-Phase2}
\end{figure}

\section{Conclusions}
In this paper a stability analysis of a  four-dimensional Cournot duopoly  model with tax evasion and time-delay in a continuous-time framework is presented. In the model analyzed, defined  as a gradient system,  two relevant parameters were detected, namely, the marginal cost rate $\mu$ and  the delay $\tau$  representing the time-delay of the second firm to enter the market after  the first firm. The parameter $\mu$ provides delay-independent and delay-dependent stability conditions. The delay-dependent stability conditions imply the existence of critical values $\tau=\tau^*$ for which the Cournot duopoly nonlinear model has limit cycles and Hopf bifurcations.

\bibliographystyle{amsplain}

\bibliography{DuopolioCNSNS}

\providecommand{\bysame}{\leavevmode\hbox to3em{\hrulefill}\thinspace}
\providecommand{\MR}{\relax\ifhmode\unskip\space\fi MR }
\providecommand{\MRhref}[2]{%
  \href{http://www.ams.org/mathscinet-getitem?mr=#1}{#2}
}
\providecommand{\href}[2]{#2}
\begin{thebibliography}{10}

\bibitem{A3}
Gian~Italo Bischi and Michael Kopel, \emph{Equilibrium selection in a nonlinear
  duopoly game with adaptive expectations}, Journal of Economic Behavior \&
  Organization \textbf{46} (2001), no.~1, 73--100.

\bibitem{A2}
Gian~Italo Bischi and Ahmad Naimzada, \emph{Global analysis of a dynamic
  duopoly game with bounded rationality}, Advances in dynamic games and
  applications, Springer, 2000, pp.~361--385.

\bibitem{cooke1986zeroes}
Kenneth~L Cooke and Pauline Van Den~Driessche, \emph{On zeroes of some
  transcendental equations}, Funkcialaj Ekvacioj \textbf{29} (1986), no.~1,
  77--90.

\bibitem{cournot1897}
Antoine~Augustin Cournot, \emph{Researches into the mathematical principles of
  the theory of wealth}, Macmillan, 1897.

\bibitem{sanjuandinamica}
Miguel~{\'A}ngel Fern{\'a}ndez~Sanju{\'a}n, \emph{Din{\'a}mica no lineal,
  teor{\'\i}a del caos y sistemas complejos: una perspectiva hist{\'o}rica},
  Rev.~R.~Acad.~Cienc.~Exact.~Fís~Nat. (2016).

\bibitem{fussmann2000crossing}
Gregor~F Fussmann, Stephen~P Ellner, Kyle~W Shertzer, and Nelson~G Hairston~Jr,
  \emph{{Crossing the Hopf bifurcation in a live predator-prey system}},
  Science \textbf{290} (2000), no.~5495, 1358--1360.

\bibitem{4}
Robert Gibbons, \emph{Un primer curso de teor{\'\i}a de juegos}, Antoni Bosch
  Editor, 1993.

\bibitem{GR}
Laszlo Goerke and Marko Runkel, \emph{Tax evasion and competition}, Scotish
  journal of political economy \textbf{58} (2011), no.~5, 711--736.

\bibitem{gori2015continuous}
Luca Gori, Luca Guerrini, and Mauro Sodini, \emph{A continuous time cournot
  duopoly with delays}, Chaos, Solitons \& Fractals \textbf{79} (2015),
  166--177.

\bibitem{hirsch2012differential}
Morris~W Hirsch, Stephen Smale, and Robert~L Devaney, \emph{Differential
  equations, dynamical systems, and an introduction to chaos}, Academic press,
  2012.

\bibitem{itza2012}
B.~A. Itz\'{a}-Ortiz and Y.~Mera-Lorenzo, \emph{Modelos de duopolio de cournot
  con evasi\'on de impuestos}, Miscel\'anea Matem\'atica \textbf{55} (2012),
  no.~1, 79--97.

\bibitem{jun2001study}
Ma~Jun-hai and Chen Yu-shu, \emph{Study for the bifurcation topological
  structure and the global complicated character of a kind of nonlinear finance
  system (ii)}, Applied Mathematics and Mechanics \textbf{22} (2001), no.~12,
  1375--1382.

\bibitem{michiels2007stability}
Wim Michiels and Silviu-Iulian Niculescu, \emph{Stability and stabilization of
  time-delay systems: an eigenvalue-based approach}, SIAM, 2007.

\bibitem{neamctu2010deterministic}
Mihaela Neam{\c{t}}u, \emph{Deterministic and stochastic cournot duopoly games
  with tax evasion}, WSEAS Transactions on Mathematics \textbf{9} (2010),
  no.~8, 618--627.

\bibitem{neimark1973d}
Ju~I Neimark, \emph{D-decomposition of the space of quasi-polynomials (on the
  stability of linearized distributive systems)}, American Mathematical Society
  Translations \textbf{102} (1973), 95--131.

\bibitem{8}
Martin~J Osborne et~al., \emph{An introduction to game theory}, Oxford
  university press New York, 2004.

\bibitem{pecora2018heterogenous}
Nicol{\`o} Pecora and Mauro Sodini, \emph{A heterogenous cournot duopoly with
  delay dynamics: Hopf bifurcations and stability switching curves},
  Communications in Nonlinear Science and Numerical Simulation \textbf{58}
  (2018), 36--46.

\bibitem{poincare1885equilibre}
Henri Poincar{\'e}, \emph{Sur l'{\'e}quilibre d'une masse fluide anim{\'e}e
  d'un mouvement de rotation}, Acta mathematica \textbf{7} (1885), no.~1,
  259--380.

\bibitem{A1}
T{\"o}nu Puu, \emph{Chaos in duopoly pricing}, Chaos, solitons, and fractals
  \textbf{1} (1991), no.~6, 573--581.

\bibitem{strogatz:2000}
Steven~H. Strogatz, \emph{Nonlinear dynamics and chaos: With applications to
  physics, biology, chemistry and engineering}, Westview Press, 2000.

\bibitem{vyhlidal2009mapping}
Tomas Vyhlidal and Pavel Z{\'\i}tek, \emph{Mapping based algorithm for
  large-scale computation of quasi-polynomial zeros}, IEEE Transactions on
  Automatic Control \textbf{54} (2009), no.~1, 171--177.

\end{thebibliography}


\end{document}